\documentclass[12pt,openany]{book}
\usepackage[margin=1in,bottom=0.5in,includefoot]{geometry}
\usepackage{amsmath,amsfonts,amsthm,amssymb,blindtext
,graphicx,pifont,fncylab,theoremref,sectsty,tocloft,titlesec,enumitem,
fmtcount,titletoc}
\counterwithin*{equation}{section}
\graphicspath{png}

\newcommand\skiplines[1]{\vspace{#1\baselineskip}}

\newtheorem{thm}{Theorem}[section]
\newtheorem{lem}{Lemma}[section]
\newtheorem{cor}{Corollary}[section]

\newtheorem{prp}{Proposition}[section]

\makeatletter
\newcommand\Dotfill{\leavevmode\cleaders\hb@xt@ .83em{\hss .\hss}\hfill \kern \z@}
\makeatother

\titleformat{\chapter}[display]
{\normalfont\normalsize\normalfont\centering}{\centering\chaptertitlename\ \thechapter}{12pt}{\normalsize}
\titlespacing*{\chapter} {0pt}{50pt}{40pt}
\setlist[description]{font=\normalfont}

\titlecontents{chapter}
  [0pt] 
  {\normalfont} 
  {\chaptername\ \thecontentslabel:\quad}
  {}
  {\Dotfill\contentspage}
\begin{document}
  \pagenumbering{gobble}
  
  \begin{titlepage}
  \begin{center}

  \vspace*{2\baselineskip}
 Boundary behavior of analytic functions and\\
 \skiplines{1}
  Approximation Theory \skiplines{3}
  
  by \skiplines{3}
  
  Spyros Pasias \skiplines{4}
  
  A thesis submitted in partial fulfillment\\ of the requirements for the degree of\\
Doctor of Philosophy\\
  with a concentration in Pure Mathematics\\
  Department of Mathematics and Statistics\\
  College of Arts and Sciences\\
  University of South Florida \skiplines{3}
  
  Major Professor: Arthur Danielyan, Ph.D.\\
  Chris Ferekides, Ph.D.\\
  Sherwin Kouchekian, Ph.D.\\
  Boris Shekhtman, Ph.D.\skiplines{2}
  
  Date of Approval:\\
  June 27, 2022 \skiplines{3}
  
  
  Copyright © 2022, Spyros Pasias
  
\end{center}

  \end{titlepage}

\baselineskip 20pt
\chapterfont{\centering}
\chapternumberfont{\normalsize}
\chaptertitlefont{\normalsize}
\titleformat{\section}
  {\normalfont\fontsize{12}{15}\normalfont}{\thesection}{1em}{}

\pagebreak

\begin{center}
\vspace*{2\baselineskip}
{Acknowledgments}\\
\skiplines{2}

\end{center}
\pagestyle{empty}

 Firstly, I would like to express my deep appreciation to my Doctorate advisor Dr. Arthur Danielyan for supervising my research in the study of analytic functions and complex approximation theory, as well as for his constant encouragement and continuous support.\\
In addition, I would like to sincerely thank all of my family, and the rest of my committee members Dr. Chris Ferekides, Dr. Sherwin Kouchekian and Dr. Boris Shekhtman who have been inspiring mentors and extremely helpful in multiple occasions during my studies at the University of South Florida.\\
   \indent Finally, I would like to express my gratitude to Dr. Dmytro Savchuk for all his assistance and valuable advise during my journey.

\pagebreak
\pagenumbering{roman}
\begin{center}

\tableofcontents

\end{center}

\pagebreak

\chapter*{Abstract}

\addcontentsline{toc}{chapter}{Abstract}
  \indent    \indent In this Thesis we deal with problems regarding boundary behavior of analytic functions and approximation theory. We will begin by characterizing the set in which Blaschke products fail to have radial limits but have unrestricted limits on its complement.\\
 \indent We will then proceed and solve several cases of an open problem posed in \cite{Da}. The goal of the problem is to unify two known theorems to create a stronger theorem; in particular we want to find necessary and sufficient conditions on sets $E_1\subset E_2$ of the unit circle such that there exists a bounded analytic function that fails to have radial limits exactly on $E_1$, but has unrestricted limits exactly on the complement of $E_2$. One of the several cases extends the main theorem proven by Peter Colwell found in [10] regarding boundary behavior of Blaschke products.\\
 \indent Additionally, we will provide a shorter proof for the necessity part of the main result in [10] which relies on a classical result proven by R. Baire. The sufficiency part of that result will then be used to shorten another proof by A.J Lohwater and G. Piranian found in [23].\\
 \indent Lastly, we will provide an extension of a well known theorem of Arakeljan about approximating continuous functions which are analytic in the interior of a closed set, by functions analytic in a larger domain.

\chapter*{Introduction}
\indent \indent The purpose of this Thesis is to examine various topics in the theory of analytic functions and approximation theory, as well as to provide some original research results that have been obtained. This manuscript is self contained, as such it may be presented to a vast audience of mathematicians. \\
\indent In Chapter $1$ we recall some fundamental concepts that will also be used in later Chapters. In particular, various material from classical books in approximation theory will be covered. Thus, the first Chapter provides knowledge on some of the fundamental concepts in complex analysis and serves as a bridge to the following Chapters where original research is to be found. \\
\indent In the following Chapter we will begin by examining a specific class of bounded analytic functions. Namely we will investigate Blaschke products; a special class of bounded analytic functions bounded by modulus by $1$. Blaschke products are an important class of functions in analytic function theory, both because of their involvement in factorization theorems and their utility for the construction of examples. We will proceed by solving a characterization theorem regarding the boundary behavior of Blaschke products. The techniques used for that result involve deep theorems about the boundary behavior of Blaschke products. In the following section we will solve several special cases of an open problem attempting to unify two known theorems regarding boundary behavior of bounded analytic functions.\\ 
\indent In the following chapter we will see how the necessity of a theorem proven in \cite{PE} regarding the boundary behavior of Blaschke products can be proven in the basis of Baire's Theorem about functions of the Baire first class. Moreover, we will see how the theorem from \cite{PE} can be utilized to simplify a proof of a theorem found in \cite{PI}.\\
\indent In the final Chapter we extend a classical theorem in approximation theory, namely Arakeljan's theorem. The methodology of the proofs in this chapter follow the methods introduced in \cite{RU}. An alternative viewpoint for the description of Arakeljan sets will be introduced. This new description will be the motivation, in which the proofs in that chapter are based on.

\pagenumbering{arabic}
\pagestyle{plain}
\chapter{Preliminaries}
\section{Measure and Integration}
 \indent \indent In this section we will recall some basic background material on measure and integration theory following the presentation in \cite{HO}.\vspace{0.12in}

 To this end, let $X$ be a set and $P(X)$ be the collection of all subsets of $X$. Then $P(X)$ is a ring under the following operations $(P(X),+,*)$ \begin{align*}
&A+B=(A\cup B) \setminus (A\cap B)\\
&A*B=A\cap B
\end{align*}

\noindent\textbf{Definition 1.1.1.}
 A subring of the ring $P(X)$ is called a $\sigma-$ring of $X$ if it is closed under countable union.\vspace{0.12in}
 
\noindent\textbf{Definition 1.1.2.}
Let $I$ be a closed interval. The Borel subsets of $I$ are the members of the smallest $\sigma-$ring of subsets of $I$ which contains all half open intervals $[a,b)\subset I$.\vspace{0.12in}

\noindent\textbf{Definition 1.1.3.}
A finite positive Borel measure on $I$ is a function $\mu:P(I)\rightarrow \mathbb{R}_{>0}$ such that for any countable collection of pairwise disjoint Borel subsets $\lbrace A_i\rbrace$ $$\mu(\bigcup_{n=1}^{\infty} A_n)=\sum_{n=1}^\infty \mu(A_n).$$

 For a closed interval $I$ every bounded variation function $F$ continuous from the left corresponds to a finite measure $\mu$. Indeed, let $[a,b)\subset I$ be a semi-closed interval, we define $$\mu([a,b))=F(b)-F(a).$$ 
\noindent Then $\mu$ has a unique extension to a Borel measure on $I$.\vspace{0.12in}

\noindent\textbf{Remark 1.1.1.}
In $I$ every Borel measure in constructed from such an increasing function $F$. Moreover, the increasing function $F(x)=x$ induces the Lebesgue measure on $I$.\vspace{0.12in}

\noindent\textbf{Remark 1.1.2.}
A finite signed measure $\mu$ on $I$ is defined to be a measure $\mu=\mu_1-\mu_2$, where $\mu_1$ and $\mu_2$ are finite positive Borel measures on $I$. Similarly, a complex Borel measure $\mu$ is a measure $\mu=\mu_1+i\mu_2$, where $\mu_1$ and $\mu_2$ are finite Borel measures.\vspace{0.12in}

\noindent\textbf{Definition 1.1.4.}
A simple function for $\mu$ is a complex valued function $f:I\rightarrow \mathbb{C}$ of the form $$f(x)=\sum_{k=1}^n a_k\chi_{E_k}(x)$$ such that:\begin{enumerate}
\item[i.] $a_1,a_2,...,a_n$ are complex numbers
\item[ii.] The sets $E_k$ are disjoint Borel subsets of $I$ of finite $\mu$ measure
\item[iii.] The function $\chi_{E_k}(x)$ denotes the characteristic function on $E_k$ 
\end{enumerate}
 For such simple functions we define $$\int_I f d\mu=\sum_{k=1}^n a_k\chi_{E_k}$$

\noindent\textbf{Remark 1.1.3.}
Clearly if $f$ is a simple function then so is $|f|$. Moreover, by the triangle inequality it is evident that $$\left|\int_I f d\mu\right|\leq\int_I \left| f\right| d\mu$$ 

\noindent\textbf{Definition 1.1.5.}
A Borel function $f:I\rightarrow \mathbb{C}$ is integrable with respect to $\mu$ if there exist a sequence $\lbrace f_n\rbrace$ such that:
\begin{enumerate}
\item[i.] Each function $f_n$ is a simple Borel function.
\item[ii.] $\lim_{m,n\rightarrow\infty}\int_I|f_n-f_m|=0.$
\item[iii.] The sequence $\lbrace f_n\rbrace$ converges in measure to $f$, that is for each $\epsilon>0$ $$\lim_{n\rightarrow\infty}\mu(\lbrace x\in I:|f_n-f|(x)\geq\epsilon\rbrace)=0.$$
\end{enumerate} 

 For every sequence of simple functions $\lbrace f_n\rbrace$ satisfying conditions $(i-iii)$ we define the integral of $f$ as: $$\int_I f d\mu=\lim_{n\rightarrow\infty}\int_I f_n d\mu.$$

\noindent\textbf{Remark 1.1.4.}
The integral of $f$ is well defined. That is, for any such sequence $\lbrace f_n\rbrace_{n=1}^\infty$ we have $\lim_{n\rightarrow\infty} \int_If_n d\mu<\infty$. Moreover, if we were to pick another sequence of simple functions $\lbrace g_n\rbrace$ satisfying conditions $(i-iii)$, then $$\int_I f d\mu=\lim_{n\rightarrow\infty}\int_I f_n d\mu=\lim_{n\rightarrow\infty}\int_I g_n d\mu.$$ 

 We denote the class of integrable functions with respect to a measure $\mu$ by $L^1(d\mu)$. We note that $f\in L^1(d\mu)$ if and only if $|f|\in L^1(d\mu)$. This is a clear difference between Riemann and Lebesgue integral. Similarly for $p>0$ we say that $f\in L^p(d\mu)$ if and only if $|f|^p$ is integrable with respect to $\mu$.\vspace{0.12in}

\noindent\textbf{Definition 1.1.6.}
A subset $S$ of $I$ is called of $\mu-$measure zero if for all $\epsilon>0$ there exists a Borel set $B\supset S$, such that $\mu(B)=0$.\vspace{0.12in}

Any phenomenon which occurs everywhere except on a set of $\mu-$measure zero is said to happen almost everywhere relative to $\mu$.\vspace{0.12in}

The following theorem is fundamental in measure and integration theory. It provides us with a great tool for determining whether a function $f$ is integrable.

\begin{thm}[Lebesgue Dominated Convergence Theroem]
If $\lbrace f_n\rbrace$ is a sequence of integrable functions such that the limit $\lim_{n\rightarrow\infty} f_n(x)=f(x)$ exists almost everywhere, and if there exists a $g\in L_1(d\mu)$ such that $|f_n|\leq g$ for all $n$, then $f$ is integrable and $$\int_If d\mu=\lim_{n\rightarrow\infty}\int_I f_n d\mu.$$
\end{thm} 

 Another fundamental result in Integration theory is \emph{H\"{o}lder's inequality}. It provides a useful tool for determining integrability and also for finding upper bounds of integrals.

\begin{thm}[H\"{o}lder's inequality]
Let $f\in L^p(d\mu)$ and $g\in L^q(d\mu)$ where $\frac{1}{p}+\frac{1}{q}=1$. Then the function $fg\in L_1(d\mu)$ and $$\int_I |fg| d\mu\leq(\int_I|f|^p d\mu)^\frac{1}{p}(\int_I|g|^q d\mu)^\frac{1}{q}.$$
\end{thm}
\section{Approximate Identities and the Poisson Kernel}
\indent \indent In this section we will look at a clever construction that has many uses in Approximation Theory. Namely we will look at approximate identities. But before we define approximate identities we recall some basic Functional Analysis.\vspace{0.12in}

\noindent\textbf{Definition 1.2.1.}
Let $X$ be a real or complex vector space. A norm on $X$ is a positive real function denoted by $\parallel...\parallel$ which has the following properties \begin{enumerate}
\item[i.] $\parallel x\parallel\geq0$; and $x=0$ if and only if $x=0$.
\item[ii.]$\parallel x+y\parallel\leq\parallel x\parallel+\parallel y\parallel$\   \ (Triangle Inequality).
\item[iii.]$\parallel \lambda x\parallel=|\lambda|\parallel x\parallel$.
\end{enumerate}

 \noindent\textbf{Definition 1.2.2.}
 A complex or real vector space $X$ endowed with a norm function $\parallel...\parallel$ becomes a metric space under the distance function $$\rho(x,y)=\parallel x-y\parallel$$ If this metric space is complete it is called a \emph{Banach space}.\vspace{0.12in}
 
 \noindent\textbf{Definition 1.2.3.} Let $I=[a,b]$ be a closed interval and let $\mu$ be a Borel measure on $I$. We define the class $L^p$ to be the equivalence class $(\sim)$ of functions $f$ such that $|f|^p$ is integrable. Moreover, we require $f\sim g$ if and only if $f-g=0$ almost everywhere. Here, we also want to include the special case $p=\infty$. In the case of $L^\infty$ the norm is defined as follows $$\parallel f \parallel_\infty\equiv\inf_x\lbrace C\geq 0: |f(x)|\leq C \text{ for almost all }x\rbrace.$$
 Just like $L^p$ we are identifying two functions in $L^\infty$ if they agree almost everywhere.\vspace{0.12in}
 
 The space $L^p$ is endowed with the $L^p$ norm $\parallel f \parallel_p=(\int_a^b|f|^p)^\frac{1}{p}$. One can check this is indeed a norm by verifying that the conditions of {Definition 1.2.1} are satisfied. Under the $L^p$ norm our space $L^p$ it is a complete metric space, hence a Banach space. That is for every sequence of functions $\lbrace f_n \rbrace_{n=1}^\infty$ of functions satisfying the Cauchy condition $$\lim_{m,n\rightarrow\infty}\parallel f_n-f_m\parallel_p=0.$$
There is an $f$ in $L^p$ such that $\parallel f-f_n\parallel_p\rightarrow 0$.
Similarly the space $L^\infty$ is a complete norm space as well, hence a Banach space.\vspace{0.12in}

 Now we note the functions $f_n$ do not necessarily converge pointwise to $f$. However, by the following theorem and a theorem of F. Riesz there is always a subsequence which converges almost everywhere. 

\begin{thm}
If $\lbrace f_n\rbrace\subset L^p$ is a sequence defined on $[a,b]$ that converges to a function $f(x)$ in the $L^p$ norm, then it converges to $f(x)$ in measure.
\end{thm}

\begin{proof}
Let $\epsilon>0$ be arbitrary and let $$E_n(\epsilon)\equiv\lbrace x\in[a,b]:|f_n(x)-f(x)|\geq\epsilon\rbrace.$$
Then $$\int_a^b(f_n-f)^2dx\geq \int_{E_n(\epsilon)}(f_n-f)^2dx\geq\epsilon^2m(E_n(\epsilon)).$$
Since $\epsilon$ is fixed, $$m(E_n(\epsilon))\rightarrow 0$$
\end{proof}

\begin{thm}[F. Riesz] Let $\lbrace f_n(x)\rbrace$ be a sequence of functions which converges in measure to the function $f(x)$. Then there exists a subsequence which converges to the function $f(x)$ almost everywhere.
\end{thm}

\begin{proof}
The proof can be found in (\cite{NA}, p. 98).
\end{proof}

\begin{cor} If the sequence $\lbrace f_n(x)\rbrace$ converges to $f(x)$ in the $L^p$ norm, then there exists a subsequence $\lbrace f_{n_k}(x)\rbrace$ which converges to $f(x)$ almost everywhere.
\end{cor}

Now let $X$ be a Banach space and consider the space $X^*$ of all linear functionals $F$ on $X$ which are continuous, that is: $$\parallel x_n-x\parallel\rightarrow 0 \Rightarrow\parallel F(x_n)-F(x)\parallel\rightarrow 0$$
Then there is a natural norm on $X^*$ which is defined as follows $$\parallel F\parallel=\sup_{\parallel x\parallel\leq 1}|F(x)|$$
With this norm $X^*$ becomes a Banach space, the conjugate space of $X$. Suppose $1\leq p<\infty$, a standard result in Functional Analysis is that the conjugate space of $L^p$ is $L^q$ where $\frac{1}{p}+\frac{1}{q}=1$. If $g\in L^q$ then it induces a linear functional $F$ on $L^p$ defined as follows: $$F(f)=\int_a^b fg dx$$
Every continuous linear functional on $L^p$ has this form, and $$\parallel F\parallel=\parallel g\parallel_q$$
The conjugate space of $L^\infty$ contains $L^1$; but except in trivial cases is larger than $L^1$.\vspace{0.12in}

\noindent\textbf{Definition 1.2.4.}
Let $H$ be a real or complex vector space. An inner product on $H$ is a function $\langle\cdot,\cdot\rangle$ which assigns to each ordered pair of vectors in $H$ a scalar, in such a way that \begin{enumerate}
\item[i.] $\langle x_1+x_2,y\rangle=\langle x_1,y\rangle + \langle x_2,y\rangle$.
\item[ii.] $\langle\lambda x,y\rangle=\lambda\langle x,y\rangle$.
\item[iii.] $\langle x,y\rangle=\overline{\langle y,x\rangle}$.
\item[iv.] $\langle x,x\rangle\geq 0; \langle x,x\rangle=0 \text{ iff }x=0$.
\end{enumerate}

 Such a space $H$, together with a specified inner product on $H$ is called an inner product space. In any inner product space one has the Cauchy-Schwartz inequality: $$ |(x,y)|^2\leq \langle x,x\rangle \langle y,y\rangle$$
This inequality is evident if $y=0$. If $y\not=0$, the inequality easily follows from\\
 $0\leq \langle x+\lambda y,x+\lambda y\rangle$ where $\lambda$ is the scalar $$\lambda=-\frac{\langle x,y\rangle}{\langle y,y\rangle}.$$
\indent From the Cauchy-Schwartz one can easily show that $\parallel x\parallel=\langle x,x\rangle^{\frac{1}{2}}$ is a norm on $H$. If $H$ is complete under this norm we say that $H$ is a \emph{Hilbert Space}. Thus, a Hilbert space is a Banach space in which the norm is induced by an inner product.\vspace{0.12in}

\noindent\textbf{Example 1.2.1.}
The space $L^2[a,b]$ with the inner product $\langle f,g\rangle=\int_a^bf\overline{g}dx$ is a Hilbert space.\vspace{0.12in}

\noindent\textbf{Definition 1.2.5.}
An approximate identity is a family of real positive functions $\lbrace k_\lambda\rbrace_{\lambda\in \Lambda\subset\mathbb{R}}$ in $L_1[-\pi,\pi]$ such that:
\begin{enumerate}
\item $\frac{1}{2\pi}\int_{-\pi}^{\pi} k_\lambda d x=1$
\item Let $I$ be any open interval about $x=0$ and $\lambda_0=\sup\lbrace \lambda:\lambda\in \Lambda\subset\mathbb{R}\rbrace$, then $$\lim_{\lambda\rightarrow\lambda_0}\sup_{x\notin I}|k_\lambda(x)|=0.$$
\end{enumerate}

 \indent Now we will look a theorem demonstrating the role of approximate identities in the theory of analytic functions.

\begin{thm}
Let $\lbrace k_\lambda\rbrace_{\lambda\in \Lambda}$ be an approximate identity. Then for all $f\in L_p[-\pi,\pi]$, $p>1$ we have: $$\lim_{\lambda\rightarrow\lambda_0} |f-\sigma_\lambda|=0$$
Where $$\sigma_\lambda (x)=\int_{-\pi}^{\pi}f(x-t)k_{\lambda}(t)dt$$ denotes the convolution of a $2\pi$ periodic extension of $f$ with $k_\lambda$. Moreover if $f$ is continuous with $f(-\pi)=f(\pi)$ then $f*k_\lambda$ converges uniformly to $f$.
\end{thm}

\begin{proof}
First we consider the case where $f$ is continuous with $f(\pi)=f(-\pi)$. Since $f$ takes on the same value on its endpoints we can extend it periodically to a continuous function on the real line. Our goal then is to show that $\lim_{\lambda\rightarrow\lambda_0}|\sigma_{\lambda}(x)-f(x)|\rightarrow 0$\\
Since for all $\lambda$ we have $\int_{-\pi}^{\pi}  k_\lambda(x) dx =1$, then for $\delta>0$ we have $$\sigma_{\lambda}(x)-f(x)=\frac{1}{2\pi}\int_{-\delta}^{\delta}[f(x-t)-f(x)]k_{\lambda}(t)dt+ \int_{|t|>\delta}[f(x-t)-f(x)]k_{\lambda}(t)dt,$$ and we see that
\begin{equation}
\tag{$\star$} |\sigma_{\lambda}(x)-f(x)|\leq \sup_{-\delta<t<\delta}|f(x-t)-f(x)|+2\parallel f\parallel_\infty \sup_{t>|\delta|}k_{\lambda}(t).
\end{equation}
Now fix $\epsilon>0$, since $f$ is continuous clearly for sufficiently small $\delta$ we have that $$\sup_{-\delta<t<\delta}|f(x-t)-f(x)<\frac{\epsilon}{2}$$ Moreover, by property $2$ in {Definition 1.2.5} we have that $\lim_{\lambda\rightarrow\lambda_0}\sup_{x\notin I}|k_\lambda(x)|=0$. Therefore, we can pick $\delta_0>0$ such that
\begin{equation*}
 \sup_{t>|\delta|}\lbrace k_{\lambda}(t)|\lambda\in (\lambda_0-\delta_0 ,\lambda_0)\rbrace<\frac{\epsilon}{2}.  
\end{equation*}
Therefore, by ($\star$) we conclude for all $\epsilon>0$ there exists sufficiently small $\delta$, such that $\lambda\in (\lambda-\delta,\lambda_0)$ implies $|\sigma_{\lambda}(x)-f(x)|<\epsilon$. The proof of uniform convergence is complete.\\
Now for $f\in L^p$ we wish to estimate $$\parallel \sigma_{\lambda}(x)-f(x)\parallel_p$$
Let $g\in L^q$ be any function, where $\frac{1}{p}+\frac{1}{q}=1$. Then 
$$
\frac{1}{2\pi}\int_{-\pi}^{\pi}[\sigma_\lambda (x)-f(x)]g(x)dx=\frac{1}{4\pi^2} \int_{-\pi}^{\pi}[f(x-t))-f(x)]g(x)k_{\lambda}(t)dxdt$$ and thus
$$\Bigg|\frac{1}{2\pi}\int_{-\pi}^{\pi}[\sigma_\lambda (x)-f(x)]g(x)dx\Bigg|\leq \frac{1}{2\pi}\int_{-\pi}^{\pi}\Bigg|\frac{1}{2\pi}\int_{-\pi}^{-\pi}[f(x-t)-f(x)]g(x)dx\Bigg| k_\lambda(t)dt
$$
Using the H\"{o}lder inequality, the inside integral is not larger in modulus than $$\parallel g\parallel_q\cdot\parallel f_t-f\parallel_p, $$
where $f_t(x)=f(x-t)$ represents the translation of $f$ by $t$ units to the right. Thus, $$\Bigg|\frac{1}{2\pi}\int_{-\pi}^{\pi}[\sigma_\lambda (x)-f(x)]g(x)dx\Bigg|\leq\parallel g\parallel_q\cdot\frac{1}{2\pi}\int_{-\pi}^{\pi}\parallel f_t-f\parallel_p k_\lambda(t)dt$$
for every $g\in L^q$. Therefore, $$\parallel\sigma_\lambda-f\parallel_p\leq\frac{1}{2\pi}\int_{-\pi}^{\pi}\parallel f_t-f\parallel_p k_\lambda(t)dt.$$
We know that $L_q$ is the conjugate space of $L^p$. Now, since of course the norm $\parallel\cdot\parallel_p$ is a linear functional in $L^p$ if we are given a function $h$ in $L^p$ we can find (by the \emph{Hahn-Banach Theorem}) a $g\in L^q$ such that $\parallel g\parallel_q=1$ and $\int hg = \parallel g\parallel_p$. Thus for $\delta>0$ we write \begin{align*}
\parallel\sigma_\lambda-f\parallel_p&\leq\frac{1}{2\pi}\int_{-\delta}^\delta\parallel f -f_t\parallel_p k_\lambda(t)dt+\frac{1}{2\pi}\int_{t>|\delta|} \parallel f_t-f\parallel_p k_{\lambda}(t)dt\\
&\leq\sup_{-\delta<t<\delta}\parallel f_t-f\parallel_p+2\parallel f\parallel_p\cdot\sup_{t\geq|\delta|}k_{\lambda}(t)
\end{align*}
Obviously the translation is a continuous operation in the $L^p$ norm, thus for small $\delta$ we have $$\lim_{\lambda\rightarrow\lambda_0}\parallel \sigma_\lambda(x)-f(x)\parallel_p=0.$$
Hence the proof is complete.
\end{proof} 

The most common example of an Approximate identity is the Poisson kernel: $$p_r(\theta)=\frac{1-r^2}{1-2rcos(\theta)+r^2}=Re(\frac{e^{i\theta}+z}{e^{i\theta}-z})$$
 From the expression is easy to check that:\begin{enumerate}
 \item[$i$.]$p_r(\theta)\geq 0$.
 \item[$ii.$]$\frac{1}{2\pi}\int_{-\pi}^{\pi}p_r(\theta)d\theta=1$, $0\leq r<1$.
 \item[$iii.$] For any open interval $I$ about $\theta=0$,\\
  $\lim_{r\rightarrow 1}\sup_{x\notin I}|p_r(\theta)|=0$.
\end{enumerate}
Therefore the Poison kernel is indeed an approximate identity.\vspace{0.1in}

The Poisson kernel is a special harmonic function. In fact it is fundamental in function theory because it can be easily shown that any given complex harmonic function in $|z|<R$ $(R>1)$ can retrieve its values in the open unit disc, just from the power series expansion of its boundary function on the unit circle. This can be achieved by taking its convolution with the Poisson kernel, noting that the power series of $p_r(\theta)=\sum_{n=-\infty}^{\infty}r^{|n|} e^{in\theta}$. Indeed, the Fourier series of a general harmonic function is $u(re^{i\theta})=\sum_{-\infty}^{\infty} c_n r^{|n|}e^{in\theta}$ and the series converges uniformly on compact subsets of $|z|<R$. Now since the convolution multiplies "pointwise" the coefficients of the Fourier series of $u(e^{i\theta})$ with coefficients of the Fourier series of the Poisson kernel $p_r(\theta)$, the result follows. Moreover, the condition $R>1$ may also be dropped since the Poisson kernel is an approximate identity. That is, any given integrable function defined on the unit circle, generates a unique harmonic function on the unit disc by taking its convolution with the Poisson kernel.

\begin{thm}
Let $f$ be a complex-valued function in $L^p$ of the unit circle, where $1\leq p<\infty$. Define $f$ in the unit disc by $$f(re^{i\theta})=\frac{1}{2\pi}\int_{-\pi}^{\pi}f(t)p_r(\theta-t)dt.$$
Then the extended function $f$ is harmonic in the open unit disc and as $r\rightarrow 1$, the functions $f_r\equiv f(re^{i\theta})$ converge to $f$ in the $L^p$ norm. If $f$ is continuous on the unit circle, the $f_r$ converge uniformly to $f$. Thus the extended $f$ is continuous on the unit disc and harmonic in its interior.
\end{thm}
 Another fundamental result in the theory of bounded analytic functions is \emph{Fatou's Theorem}. Its proof can be deduced simply by using approximate identities (\cite{HO}, p. 30-33) and a well known theorem of \emph{Lindel\"{o}f}.
 
 \begin{thm}[Fatou's Theorem]
 Let $\mu$ be a finite complex measure on the unit circle $T$, and let $f$ be the harmonic function in the unit disc defined by $$f(r,\theta)=\int_{-\pi}^{\pi}p_r(\theta-t)d\mu(t).$$
 Let $\theta_0$ be any point where $\mu$ is differentiable with respect to Lebesgue measure. Then $$\lim_{r\rightarrow 1}f(r,\theta_0)=2\pi(\frac{d\mu}{d\theta})(\theta_0)=2\pi\mu'(\theta_0)$$
 \end{thm}
 
 \begin{proof}
Let $F:[-\pi,\pi]\rightarrow\mathbb{C}$ be the complex valued function of bounded variation that induces $\mu$. That is, for all $g$ which are integrable with respect to our measure $\mu$ $$\int_{-\pi}^{\pi}gd\mu=\int_{-\pi}^{\pi}gdF.$$
Therefore, we must show that if $F$ is differentiable at $\theta_0$, then $$\lim_{r\rightarrow 1}f(r,\theta_0)=F'(\theta_0)$$
as $re^{i\theta}\rightarrow e^{i\theta_0}$ in a non tangential manner.\\
 Firstly, we observe that {Fatou's Theorem} is trivially true for the case where $\mu$ is the Lebesgue measure $d\theta$. Therefore, without loss of generality (by subtracting a suitable constant if necessary) we may assume that $\mu(T)=0$; so that $F(-\pi)=F(\pi)$.\\
 Now let $$ F(r,\theta)=\frac{1}{2\pi}\int_{-\pi}^{\pi} P(r,\theta-t)F(t)dt.$$
Then integration by parts yields, $$\frac{1}{2\pi}\int_{-\pi}^{\pi}
p'_r(\theta-t)F(t)dt=-\frac{1}{2\pi}\Big[ p_r(\theta-t)F(t)dt
\Big]_{-\pi}^{\pi}+\frac{1}{2\pi}\int_{-\pi}^{\pi}p_r(\theta-t)dF(t)$$ 
Since without loss of generality we may assume $F$ is periodic, $$\frac{1}{2\pi}f(r,\theta)=\frac{1}{2\pi}\int_{-\pi}^{\pi}p'_r(\theta-t)F(t)dt.$$
Hence, \begin{align*}
f(r,\theta)&=\int_{-\pi}^{\pi}p'_r(\theta-t)F(t)dt\\
&=\int_{-\pi}^{0}p'_r(\theta-t)F(t)dt+\int_{0}^{\pi}p'_r(\theta-t)F(t)dt\\
&=\int_{0}^{\pi}p'_r(t)[F(\theta-t)-F(\theta+t)]dt\\
&=\int_{0}^{\pi}[-\sin(t)p'_r(t)]\frac{F(\theta-t)-F(\theta+t)}{\sin(t)}dt.
\end{align*}
Now we note that since the Poisson kernel is clearly an even function, its derivative $p'(r,t)$ must be odd. Thus we have, $$\frac{1}{2\pi}f(r,\theta)=\frac{r}{2\pi}\int_{-\pi}^{\pi}K_r(t)\frac{F(\theta+t)-F(\theta-t)}{2 \sin (t)}dt$$
where we denote, $$K_r(t)\equiv-\frac{1}{r}\sin (t)p'_r(t).$$
One may easily verify that the functions $\{K_r\}$ form an approximate identity on $L_1[-\pi,\pi ].$ Moreover, if $F$ is differentiable at $\theta_0$, then the function $$G(t)=\frac{F(\theta_0+t)-F(\theta_0-t)}{2 \sin(t)}$$
is continuous at $t=0$ and takes the value $$G(0)=\lim_{t\rightarrow 0}\frac{F(\theta_0+t)-F(\theta_0-t)}{2 \sin (t)}=F'(\theta_0).$$
Now, since the functions $\{K_r\}$ form an approximate identity and since $G$ is continuous at $t=0$ it follows \begin{align*}
\frac{1}{2\pi}f(r,\theta_0)&=\lim_{r\rightarrow 1}\int_{-\pi}^{\pi}K_r(t)G(t)dt\\
&=G(0)\\
&=F'(\theta_0).
\end{align*}
Hence the radial convergence is proved. \\
 \indent For the non tangential convergence we refer to {Lindel\"{o}f's Theorem} which states the following:\\ Let $f$ be a bounded analytic function on the unit disc; then if the limit of $f$ exists along some curve terminating to a point $e^{i\theta_0}$ (in particular radial convergence), then the non tangential limit exist at that point. For a proof of the non tangential convergence without utilizing Lindel\"{o}f's theorem one may refer to (\cite{HO}, p. 36-37). The proof is complete.
 \end{proof}
 {Fatou's Theorem} implies some important corollaries. Firstly, we note that by \emph{Lebesgue's decomposition Theorem} every finite measure $\mu$ can be decomposed to an absolutely continuous part with respect to the Lebesgue measure and a singular part, that is $$d\mu=fd\theta+d\mu_s$$ where $f$ is Lebesgue integrable and $\mu_s$ a singular measure. Then $\mu$ is differentiable almost everywhere and $\frac{d\mu}{d\theta}=\frac{1}{2\pi}f$ almost everywhere. Therefore, if we apply {Fatou's Theorem} it follows that the Poisson integral of a finite measure has non tangential limits almost everywhere with respect to the normalized Lebesgue measure.
  
 \begin{cor} Let $f$ be a complex valued harmonic function in the unit disc and suppose that the integrals $$\int_{-\pi}^{\pi}|f(re^{i\theta})|^pd\theta$$ are bounded as $r\rightarrow 1$ for some $p$, $1\leq p<\infty$. Then for almost all $\theta$ the radial limits $$\tilde{f}(\theta)=\lim_{r\rightarrow 1}f(re^{i\theta})$$ exist and define a function $\tilde{f}$ in $L^p$ of the circle. If $p>1$ then $f$ is the Poisson integral of a unique finite measure whose absolutely continuous part is $\frac{1}{2\pi}\tilde{f}d\theta$. If $f$ is a bounded harmonic function, the boundary values exist almost everywhere and define a bounded measurable function $\tilde{f}$ whose Poisson integral is $f$.
  \end{cor}
  
 \noindent\textbf{Remark 1.2.1.} The uniqueness part of the fact that if $p>1$, $f$ can be expressed as the Poisson integral of a unique finite measure is usually referred to as \emph{Herglotz representation Theorem} and we will provide the proof bellow.
  
  \begin{thm}
  Let $U(z)$ be a harmonic function on the unit disc which can be written as $$U(z)=\int_{-\pi}^{\pi} p(r,\theta-t)d\mu(t)$$
  Then $\mu(t)$ is essentially unique.
  \end{thm}
  \begin{proof}
  First recall that every measure $d\mu$ on the circle is generated by a bounded variation function $\mu$ in the following manner. $\mu[e^{ia},e^{ib}]=\mu(e^{ia})-\mu(e^{ib})$. Now, suppose there exist two functions of bounded variation, $\mu_1$ and $\mu_2$ such that $$U(z)=\int_{-\pi}^{\pi}p(r,\theta-t)d\mu_k(t)\  \ k=1,2.$$ Then we claim $\omega(t)\equiv \mu_1(t)-\mu_2(t)$ is constant almost everywhere. Clearly $\omega(t)$ is of bounded variation since $\mu_1$ and $\mu_2$ are. Hence $0=\int_{-\pi}^{\pi}p(r,\theta-t)d\omega(t)$ is well defined. Recall that the Poisson kernel $$p(r,t)=Re\left(\frac{e^{it}+z}{e^{it}-z}\right)$$ for $|z|<1$. Now, since the analytic completion of a zero harmonic function is a constant function it follows $$\int_{-\pi}^{\pi}\frac{e^{it}+z}{e^{it}-z}d\omega(t)=i\gamma,\ \ \gamma\in \mathbb{R}$$
  Hence if we set $z=0$ we get $\omega(\pi)-\omega(-\pi)=i\gamma$. But since the left hand side is real it follows $\gamma=0$.\\
  Now, $$\frac{e^{it}+z}{e^{it}-z}=1+2\sum_{k=1}^\infty(ze^{-it})^k$$Hence, \begin{align*}
  &\int_{-\pi}^{\pi} \left[1+2\sum_{k=1}^\infty z^ke^{-ikt}\right]d\omega(t)=0\\
  \Rightarrow &[\omega(t)]_{-\pi}^{\pi}+2\sum_{k=1}^\infty z^k\int_{-\pi}^{\pi}e^{-ikt}d\omega(t)=0\\
  \Rightarrow &\int_{-\pi}^{\pi}e^{-ikt}d\omega(t)=0, \ \ k\in\mathbb{N}
  \end{align*}
  Now, since the conjugate of a Stiltjes integral is the Stiltjes integral of the conjugate it follows that $\int_{-\pi}^{\pi}e^{ikt}d\omega(t)=0$, $k\in\mathbb{Z}$.\\
  Let $k\not=0$ and integrate by parts $$0=\int_{-\pi}^{\pi}e^{ikt}d\omega(t)=[\omega(t)e^{ikt}]_{-\pi}^{\pi}-\frac{1}{ik}\int_{-\pi}^{\pi}\omega(t)e^{ikt}dt$$
  It follows $\int_{-\pi}^{\pi}\omega(t)e^{ikt}dt=0$. Now denote $c=\int_{-\pi}^{\pi}\omega(t)dt$. Then, $$\int_{-\pi}^{\pi}[\omega(t)-c]e^{-ikt}dt=\int_{-\pi}^{\pi}\omega(t)e^{-ikt}dt=-c\int_{-\pi}^{\pi}e^{ikt}dt=0$$
  Now since $\omega(t)-c\in L_2[-\pi,\pi]$ it follows by \emph{Parseval's Identity} that $\int_{-\pi}^{\pi}|\omega(t)-c|^2dt=0$. Consequently, we must have that $\omega(t)=c$ at all points of continuity. However, since $\omega(t)$ is a function of bounded variation; it follows that it is continuous almost everywhere, hence the proof is complete.
  \end{proof}
  
 Our results about harmonic functions also apply in the case of analytic functions which are a particular case of complex harmonic functions. For $0< p\leq\infty$ denote by $H^p$ the class of analytic functions $f$ in the unit disc for which the functions $f_r(\theta)=f(re^{i\theta})$ are bounded in the $L^p$ norm as $r\rightarrow 1$. If $1\leq p\leq\infty$, then $H^p$ is a Banach space under the norm $$\parallel f\parallel=\lim_{r\rightarrow 1}\parallel f_r\parallel_p.$$

 Now we will proceed to see that every function in $H^p$ can be factorized uniquely into the product of three different types of functions. This factorization theorem is indeed of great significance because it allows us to analyze a function in $H^p$ by looking at each of its factors independently. In fact in the next chapter we will attempt to solve a problem proposed in $\cite{Da}$ for one of the three types of functions comprising the factorization, namely Blaschke products. One might then continue the research by examining the problem under the lens of the remaining types of functions in the factorization theorem. Hopefully, by examining the problem with each type of functions one will indeed manage to solve the problem for a general function in $H^p$.\vspace{0.12in}

\noindent\textbf{Definition 1.2.6.}
An \emph{inner function} is an analytic function $g$ in the unit disc, such that $|g(z)|\leq 1$ and $|g(e^{i\theta}|=1$ almost everywhere on the unit circle. An \emph{outer function} is an analytic function $F$ in the unit disc of the form $$F(z)=\lambda \exp\left[ \frac{1}{2\pi}\int_{-\pi}^{\pi}\frac{e^{i\theta}+z}{e^{i\theta}-z}k(\theta)d\theta\right]$$
where $k$ is a real valued integrable function on the circle and $\lambda$ is a complex number of modulus $1$.

\begin{thm}
Let $f$ be a non zero function in $H^1$. Then $f$ can be written in the form $f=gF$ where $g$ is an inner function and $F$ an outer function. Moreover, this factorization is unique up to a constant of modulus $1$.
\end{thm}

\chapter{Blaschke products and Boundary behavior problems}

\section{Blaschke products}
\indent \indent Now we are going to decompose each inner function in the product of two particular cases of inner functions. The first factor is called Blaschke products and is the factor containing all the zeros of our inner function. The second factor is determined by a measure singular with respect to the Lebesgue measure on the unit circle. We note that many of the results in this chapter are taken from the recent paper \cite{SP}.\vspace{0.12in}

 Below $D$ and $T$ denote the open unit disk and the unit circle in the complex plane respectively. Let $m$ be the Lebesgue measure on $T$. The Blaschke products are forming an important subclass of the well known space $H^\infty$ of all bounded analytic functions in $D$. We begin this section by giving motivation for the definition of Blaschke products and then proceed to answer some questions regarding the boundary behaviors of Blaschke products.

\begin{thm}
Let $f$ be a bounded analytic function in $D$ and suppose $f(0)\not=0$. If $\lbrace a_n\rbrace$ is the sequence of zeros of $f$ in the $D$, each repeated as often as the multiplicity of the zero of $f$, then the product $\prod_n|a_n|$ is convergent, that is, $$\sum_n1-|a_n|<\infty.$$
\end{thm}
\begin{proof}
The proof is simple and can be found in (\cite{HO}, p. 63-64).
\end{proof}

 Our goal now is to define an inner function having the zeros of $f$, so that we can divide by $f$ to obtain a zero free inner function. The most intuitive construction would be to define the infinite product $$\prod_n\frac{z-a_n}{1-\overline{a_n}z}.$$ Unfortunately this infinite product does not converge necessarily. However, if we simply rotate each term by a factor of $-\frac{\overline{a_n}}{|a_n|}$ then the infinite product converges. For the proof see (\cite{HO}, p. 64-65).
The above motivates the following definition for Blaschke products.\vspace{0.12in}

\noindent\textbf{Definition 2.1.1.}
Let $\lbrace a_n \rbrace$ be a sequence of complex numbers in the open unit disc, satisfying the condition $\sum_n1-|a_n|<\infty$, then the infinite product $$B(z)\equiv \prod_{n=1}^\infty-\frac{\overline{a_n}}{|a_n|}\frac{z-a_n}{1-\overline{a_n}z}$$ is well defined in the unit disc and is called a \emph{Blashcke product}.

The following well known theorem regarding boundary behavior of Blaschke products in the unit disc $D$ can be found in (\cite{HO}, p. 68). The theorem is actually concerned with convergence in the whole complex plane $\mathbb{C}$ but for the purpose of this Thesis we will restrict the Theorem just in the unit disc $D$.
\begin{thm}
If $B(z)$ is a Blaschke product with zeros $\lbrace a_n\rbrace$, then $B(z)$ converges uniformly on compact subsets of $D$. Moreover, the unrestricted limits of $B(z)$ exist everywhere on $T$ and are of modulus $1$, except at the closed set $K$ consisting of all points in $T$ that are accumulation points of $\lbrace a_n\rbrace$. The unrestricted limits of $B(z)$ fail to exist on $K$.
\end{thm}

Now we will see that a Blaschke product is indeed an Inner function.

\begin{thm}[Riesz Theorem]
The moduli of the radial limits of a Blaschke product $B$ are equal to $1$ almost everywhere on $T$.
\end{thm}

\begin{proof}
Let $$B_n(z)=\prod_{k=1}^n\frac{\overline{a_k}}{|a_k|}\frac{a_k-z}{1-\overline{a_k}z}$$ be the $n$-th partial product of our Blaschke product $B(z)$. We will first show that $\lbrace B_n(e^{i\theta})\rbrace_{n=1}^\infty$ converges in $H^2$ on the boundary circle. Indeed,
\begin{align*}
\frac{1}{2\pi}\int_{-\pi}^{\pi} |B_m(e^{i\theta})-B_n(e^{i\theta})|^2d\theta=\frac{1}{2\pi}\int_{-\pi}^{\pi} [|B_m(e^{i\theta})|^2+|B_n(e^{i\theta})|^2-2Re(B_n(e^{i\theta})\overline{B_m(e^{i\theta}})]d\theta. 
\end{align*}
 Now, each finite Blaschke product has modulus $1$ everywhere on the boundary circle $T$. It follows that for all $\theta\in [-\pi,\pi]$ we have $|B_m(e^{i\theta})|=|B_n(e^{i\theta})|=1$ and $\overline{B_m(e^{i\theta})}=\frac{1}{B_m(e^{i\theta})}$, thus: $$\frac{1}{2\pi}\int_{-\pi}^{\pi} |B_m(e^{i\theta})-B_n(e^{i\theta})|^2d\theta=\frac{1}{2\pi}\int_{-\pi}^{\pi}2-2Re\frac{B_n(e^{i\theta})}{B_m(e^{i\theta})}d\theta. $$
Now we note that for $n>m$ the function $(\frac{B_n}{B_m})(z)$ is analytic, thus by Cauchy integral formula we have: $$\frac{1}{2\pi}\int_{-\pi}^{\pi}\frac{B_n(e^{i\theta})}{B_m(e^{i\theta})}d\theta=\frac{B_n}{B_m}(0)=\prod_{k=m+1}^n|a_k|$$
Therefore, it follows  $$\frac{1}{2\pi}\int_{-\pi}^{\pi} |B_m(e^{i\theta})-B_n(e^{i\theta})|^2d\theta=2(1-\prod_{m+1}^n|a_k|).$$
Now, since $\prod_{k=1}^\infty |a_k|$ converges it follows $\lim_{m\rightarrow\infty}\prod|a_k|_{k=m+1}^\infty=1$, hence the sequence $\lbrace B_k(e^{i\theta})\rbrace_{k=1}^\infty$ is clearly Cauchy in the $L^2$ norm. Since we are in $H^2(T)$ which is a well known Hilbert space, it follows that $B_n(e^{i\theta})\rightarrow B(e^{i\theta})$ in $H^2(T)$, where of course $B(e^{i\theta})$ is the extension of $B(z)$ on $T$ by its radial values which exist almost everywhere by Fatou's Theorem. We have that $$\lim_{n\rightarrow\infty}\int_{-\pi}^{\pi}|B_n(e^{i\theta})-B(e^{i\theta})|^2d\theta=0$$
Hence by {Corollary 1.2.1} there exists a subsequence of $B_n(e^{i\theta})$ that converges pointwise to $B(e^{i\theta})$ almost everywhere. Consequently $B(e^{i\theta})$ has modulus $1$ for almost $\theta\in[0,2\pi]$ and the proof is complete. 
\end{proof}

 The following theorems can be found in (\cite{HO}, p. 66-67) and they are the final ingredients to the Factorization Theorem which will follow.
\begin{thm}
Let $f$ be a non zero bounded analytic function in the unit disc. Then $f$ is uniquely expressible in the form $f=Bg$ where $B$ is a Blaschke product and $g$ a bounded analytic function without zeros.
\end{thm}

\begin{thm}
Let $g$ be an inner function without zeros, and suppose that $g(0)$ is positive. Then there is a unique singular positive measure $\mu$ on the unit circle such that $$g(z)=exp\left[-\int_0^{2\pi}\frac{e^{i\theta}+z}{e^{i\theta}-z}d\mu(\theta)\right]$$
\end{thm}

Now we are ready to state and prove Factorization Theorem.

\begin{thm}[Factorization Theorem]
Let $f\not=0$ be an $H^1$ function in the unit disc. Then $f$ is uniquely expressible in the form $f=BSF$, where $B$ is a Blaschke product, $S$ a singular function and $F$ is an outer function (in $H^1$) .
\end{thm}
\begin{proof}
 We know by {Theorem 1.2.7} that $f=gF$, where $g$ is an inner function and $F$ is an outer function, and that this factorization is unique up to a constant of modulus $1$. If $B$ is the Blaschke product formed by the zeros of $f$, then $g=BS$, where $S$ is a zero free inner function. By multiplying $g$ with a constant of modulus $1$ we may assume without loss of generality that $S(0)>0$, and by the theorem above it follows that $S$ is a singular function. The multiplication constant can be absorbed in the outer function $F$. The proof is complete.
\end{proof}

 We saw earlier {Fatou's Theorem} which plays fundamental role in the theory of boundary behavior of analytic functions. For the purpose of what follows we will restate it in the following more compact form.

\begin{thm}[Fatou's Theorem]
Let $f\in H^\infty$. Then the radial limits of $f$ exist on $T$ except perhaps for a subset $E$ of measure zero.
\end{thm}

 It is well known that the exceptional set $E$ in {Fatou's Theorem} is a $G_{\delta\sigma}$ set. But if we assume that the function $f$ has unrestricted limits at the points of the set $T\setminus E$, then $E$ necessarily becomes an $F_\sigma$ set. The converse statement was proved in \cite{Da} on the basis of the method of \cite{AD}, and the result can be formulated as follows.

\begin{thm}
Let $E$ be a set on $T$. Then there exists a function $f\in H^\infty$ which has no radial limits on $E$ but has unrestricted limit at each point of $T\setminus E$ if and only if $E$ is an $F_\sigma$ set of measure zero.
\end{thm}

 Note that the proof of {Theorem 2.1.8} is elementary and its sufficiency part extends a well known theorem of Lohwater and Piranian \cite{PI} which states that for any $F_\sigma\subset T$ subset of measure zero there exists a bounded analytic function which fails to have radial limits exactly on that set.
Our purpose now is to state and prove the analogue of {Theorem 2.1.8} for Blaschke products. The result is the following:

\begin{thm}[Characterization Theorem]
Let E be a set on $T$. Then there exists a Blaschke product which has no radial limits on $E$ but has unrestricted limit at each point of $T\setminus E$ if and only if $E$ is a closed set of measure zero.
\end{thm}

The proof of {Theorem 2.1.9} uses {Theorem 2.1.8} and some results on the boundary behavior of Blaschke products due to R.D. Berman \cite{BE} and A. Nicolau \cite{NI}.
\section{Some auxiliary results}

\indent \indent We follow the presentation of Nicolau \cite{NI} to formulate the results of Berman \cite{BE} and Nicolau \cite{NI}, respectively. The next theorem is due to Berman (\cite{BE}, p. 250).

\begin{thm}[R.Berman] Let $E$ be a subset of the unit circle of zero Lebesgue measure and of type $F_\sigma$ and $G_\delta$. Then there exist Blaschke products $B_0$ and $B_1$ such that:\\
(i) $B_0$ extends analytically to $T\setminus \overline{E}$ and $\lim_{r\rightarrow 1} B_0(re^{it})=0$ if and only if $e^{it}\in E$;\\
(ii) $\lim_{r\rightarrow 1} B_1(re^{it})=1$ if and only if $e^{it}\in E$.  
\end{thm}
 Now we formulate a theorem of Nicolau (see \cite{NI}, Proposition on p. 251).
\begin{thm}[A. Nicolau]
Let $E$ be a subset of the unit circle. Assume that there exist a Blaschke product $B_0$ that extends analytically to $T\setminus \overline{E}$ with $\lim_{r\rightarrow 1}B_0(re^{it})=0$ for $e^{it}\in E$, and an analytic function $f_1$ in the unit ball of $H^\infty$, $f_1\neq 1$, such that $\lim_{r\rightarrow 1}f_1(re^{it})=1$ for $e^{it}\in E$. Then for each analytic function $g$ in the unit ball of $H^\infty$, there exists a Blaschke product $I$ that extends analytically to $T\setminus \overline{E}$, such that $$\lim_{r\rightarrow 1}[I(re^{it})-g(re^{it})]=0 \text{ for } e^{it}\in E$$
\end{thm}

 In this result, of course, the set $E$ in fact is of Lebesgue measure zero since if $m(E)>0$, then already the condition $\lim_{r\rightarrow 1} B_0(re^{it})=0$ for $e^{it}\in E$ will imply that $B_0\equiv 0$, which is impossible as $B_0$ is a Blaschke product.\\
As noted in (\cite{NI}, p. 251), {Theorem 2.2.1} shows that the hypothesis of {Theorem 2.2.2} are satisfied if the set $E$ is a subset of $T$ of Lebesgue measure zero and of type $F_\sigma$ and $G_\delta$.
\section{Proof of Characterization Theorem 2.1.9}
\begin{proof}
\textbf{(Necessity)} Assume that there exists a Blaschke product $B(z)$ such that it has unrestricted limits at each point of $T\setminus E$ and does not have radial limits at the points of $E$. Firstly, by {Fatou's Theorem} we have that $m(E)=0$. To prove that $E$ is closed, let $\lbrace a_n\rbrace\subset D$ be the sequence of zeros of $B(z).$ We complete the proof in two steps: \\
\underline{Step 1.} If the unrestricted limit of $B$ at $z_0\in T$ exists and is equal to $d$, then $|d|=1$. Indeed by {Theorem 2.1.3} (Riesz Theorem) in each neighborhood of $z_0$ on $T$ there exists a point $w$, at which the radial limit $B(w)$, of $B(z)$ exists and has modulus $1$. This implies $|d|=1$\\
\underline{Step 2}. The set $T\setminus E$ is open.\\
Indeed, if $z_0\in T\setminus E$, then $z_0\notin\overline{\lbrace a_n\rbrace}$ because otherwise the unrestricted limit of $B$ at $z_0$ would be equal to $0$, which contradicts step $1$. Thus, there exists an open interval $I\subset T$ such that $z_0\in I$ and $I$ is disjoint of the (closed) set $\overline{\lbrace a_n\rbrace}$. Then by Theorem 2.1.2, the Blaschke product $B(z)$ is continuous on $I$, hence it possesses unrestricted limits at the points of $I$. This means that $I\subset T\setminus E$, thus $T\setminus E$ is open. Therefore, $E$ is closed and the proof of necessity is complete.\\
\textbf{(Sufficiency)} Let $E$ be a closed set of measure zero on $T$. Since $E$ is closed , it is of type $F_\sigma$ and $G_\delta$. Thus the hypothesis of {Theorem 2.2.1} (Berman's Theorem) are satisfied for $E$. It follows that there exists Blaschke products $B_0$ and $B_1$ with properties (i) and (ii) as listed in Berman's Theorem.\\
This means that the hypothesis of {Theorem 2.2.2} (Nicolau's Theorem) indicated in the second sentence of it's formulation are satisfied. Note that the hypothesis that there exists an analytic function $f_1$ in the unit ball of $H^\infty$, $f_1\neq 1$ such that $\lim_{r\rightarrow 1}f_1(re^{it})=1$ for $e^{it}\in E$, is satisfied regardless of {Berman's Theorem}; since by Fatou's interpolation theorem such an $f_1$ exists even in the \emph{disc algebra} which consists all of continuous functions on $\overline{D}$ that are analytic in $D$.\\
By the sufficiency part of {Theorem 2.1.8}, there exists a function $g\in H^\infty$ which has no radial limits at each point of $E$ but has unrestricted limit at each point of $T\setminus E$ (obviously, $g$ can be extended continuously on the open set $T\setminus E$). By dividing to a constant if needed, we may assume that the function $g$ is in the unit ball of $H^\infty$ and thus one can apply Nicolau's Theorem. Therefore we can find a Blaschke product $I$ that extends analytically on $T\setminus E$, such that for $e^{it}\in E$ we have $$\lim_{r\rightarrow 1}[I(re^{it})-g(re^{it})]=0.$$ Since $\lim_{r\rightarrow 1}g(re^{it})$ does not exist for $e^{it}\in E$, $\lim_{r\rightarrow 1} I(re^{it})$ does not exist for $e^{it}\in E$. Since $I$ extends analytically on $T\setminus E$, it had unrestricted limits on $T\setminus E$, and the Blaschke product $I$ has the desired properties. This completes the proof of sufficiency and the theorem in proved.
\end{proof}

 Now lets formulate Step $1$ of the proof of necessity of {Theorem 2.1.9} as a proposition.
\begin{prp} If the unrestricted limit of a Blaschke product $B(z)
$ at a point $\zeta \in T$ exists and is equal to $d$, then $|d|=1$.
\end{prp}
 The following proof of {Proposition 2.3.1} repeats the proof of Step $1$.

\begin{proof}
If $N(\zeta)$ is any open disc centered at $\zeta$, {Theorem 2.1.3} provides the existence of a point $t\in D\cap N(\zeta)$ such that $|f(t)|$ is arbitrarily close to $1$. This implies that $|d|=1$
\end{proof} 

 Thus, {Proposition 2.3.1}(Step $1$) is a trivial corollary of {Riesz Theorem}. However, {Proposition 2.3.1} implies a result of which the original proof is not trivial yet uses Riesz Theorem as well. This may bring to the conclusion that some authors of the past working in the theory of cluster sets missed to completely appreciate the methodological power of {Riesz Theorem}.\\
\indent Let $f$ be a function defined in $D$ and let $\zeta \in T$. The cluster set $C(f,\zeta)$ of $f$ at $\zeta$, by definition, is the set of all numbers $w$ such that there exist a sequence $\lbrace z_n\rbrace$ in $D$ with the properties $\lim_{n\rightarrow\infty} z_n=\zeta$ and $\lim_{n\rightarrow\infty} f(z_n)=w$ (e.g. \cite{NO}, p. 52 or \cite{EG}).\\
H.G. Eggleston \cite{EG} considered the following example (see example (A) in \cite{EG}).\\

\noindent (A) \textit{Let f be a Blaschke product whose zeros are distributed in $D$ such that each point of $T$ is an accumulation point of zeros. Then for every point $\zeta\in T$ the cluster set $C(f,\zeta)$ of $f$ at $\zeta$ does not reduce to a single point} (see \cite{EG}, p. 140).\\

The expressions ``the cluster set at $\zeta$ reduces to a single point $d$" and ``the unrestricted limit at $\zeta$ equals $d$" are identical. However, in example (A) the number $d$, if it exists must be $0$, while $|d|=1$ by {Proposition 2.3.1}. This proves example (A).\\
\indent Eggleston \cite{EG} has even proved that for all $\zeta\in T$ the diameter of $C(f,\zeta)$ is greater than or equal to unity.
This conclusion of Eggleston has been carried further by E.F. Collingwood (See \cite{CO}, p. 378 and the last footnote on the same page). Collingwood has also dedicated another work (see \cite{COL}) to the same paper of Eggleston \cite{EG}. But the above obvious proofs have remained unnoticed (which perhaps explains why the necessity part of {Theorem 2.1.9} has never been stated and proved before).

\section{Unification problem}

 \indent In this section we start by stating some important theorems. In the theorems that follow in this section and on, we note that each function of the unit disc is extended to take the value of its radial limits on the unit circle $T$, when they exist. The following important theorem is due to S.V Kolesnikov \cite{KO}.

\begin{thm} Let $E\subset T$ be of type $G_{\delta\sigma}$ and of measure $0$. Then there exists a bounded analytic function in the unit disc $D$ which fails to have radial limits exactly at the points of $E$.
\end{thm}

 It should be noted that the converse statement of {Kolesnikov's Theorem} which states:\\
\noindent The set in which the radial limits of a bounded analytic function is necessarily $G_{\delta\sigma}$, is elementary and has been known long before {Kolesnikov's Theorem} was introduced. A proof can be found in (\cite{AL}, p. 23).

\begin{thm}[Fatou's interpolation theorem, 1906] Let $K$ be a closed subset of $T$ such that $m(K)=0$. Then there exists a function $f$ in the disc algebra which vanishes precisely on $K$.
\end{thm}

 A proof of {Fatou's interpolation theorem}, (that shows the real part of $f$ is positive on $D$) can be found in (\cite{KO}, p. $29-30$). A new proof may also be found in \cite{ART}. The construction of Fatou's function is also used to prove an important result by \emph{F. and M. Riesz} which states that any analytic measure is absolutely continuous with respect to the Lebesgue measure.\\
 \indent The necessity part of the following theorem is well known and its proof is elementary. The sufficiency part is proven in \cite{Da} and its proof is using the method of the paper \cite{AD} and it is based on {Fatou's interpolation theorem}.

\begin{thm}
Let $E$ be a set on $T$. Then there exists a function $f\in H^\infty$ which has no radial limits on $E$ but has unrestricted limits at each point of $T\setminus E$ if and only if $E$ is an $F_\sigma$ set of measure zero.
\end{thm}

 The following theorem is a solution to a problem proposed by Rubel in 1973 and was solved in \cite{DA} with an affirmative answer to the problem. The problem asks whether for any $G_\delta$ set $F\subset T$ of measure zero, there exists a non vanishing bounded analytic function on $D$ such that $f=0$ precisely of $F$ and such that the radial limits of $f$ exist everywhere on $T$. The following stronger theorem found in \cite{art} provides an affirmative answer to Rubel's problem as well.

\begin{thm} Let $F$ be a $G_\delta$ of measure zero on $T$. Then there exists a nonvanishing bounded analytic function $g$ such that:\begin{enumerate}
 \item $g$ has non zero radial limits everywhere on $T\setminus F$.
 \item $g$ has vanishing unrestricted limits at each point of $F$.
 \end{enumerate}
\end{thm}

 Now we use the three theorems above to solve a particular case of Problem $1$ stated in \cite{Da}. Although the problem stated in \cite{Da} is still an open problem and a good research opportunity, we have solved several particular cases of it in this thesis. The problem basically is an attempt to unify both {Kolesnikov's Theorem} and {Theorem 2.4.3} to a single stronger theorem. The problem we are interested in (See [11], p. 4) is the following:\vspace{0.12in}

\noindent\textbf{Unification Problem.}
\emph{Let $E_1\subset E_2$ be subsets of the unit circle $T\equiv\lbrace z\mid |z|=1 \rbrace$. Find necessary and sufficient conditions for the existence of an $f\in H^\infty$ such that the radial limits of $f$ fail to exist precisely on $E_1$ and it has unrestricted limits precisely at $T\setminus E_2$.}\\

 The case $E_1=\emptyset$ is also dealt with by the following theorem.

\begin{thm}[Brown, Gauthier and Hengartner]
Let $E$ be a set on $T$. Then there exists a function $f\in H^\infty$ which has no unrestricted limits on $E$ but has unrestricted limits at every point of $T\setminus E$ if and only if $E$ is of type $F_\sigma$.
\end{thm}

 The proof of the theorem above is not difficult. The idea is to express $E$ as a countable union of closed sets $E=\bigcup_n C_n$, and then for each of the sets $C_n$ pick a respective Blaschke product $B_n$ with zeros accumulating precisely on $C_n$. Then, since Blaschke products are bounded by modulus by $1$, the infinite sum $\sum_{n=1}^\infty \frac{B_n}{2^n}$ converges to the function with the desired properties.\vspace{0.12in}

 Now we will proceed and solve some particular cases of the Unification Problem. In the proofs that follow, we note that each function of the unit disc is extended to take the value of its radial limits on the unit circle $T$, when they exist.

\begin{thm} Let $E_1\subset E_2$ be subsets of the unit circle $T\equiv\lbrace z\mid |z|=1 \rbrace$ and suppose that the measure of $E_2$ is $2\pi$. The following conditions are necessary and sufficient for the existence of an $f\in H^\infty$ such that the radial limits of $f$ fail to exist precisely at $E_1$ and it has vanishing unrestricted limits precisely at $T\setminus E_2$:
\begin{enumerate}
\item $E_1$ is of type $G_{\delta\sigma}$ and of measure zero.
\item $E_2$ is of type $F_\sigma$.
\end{enumerate}
\end{thm}

\begin{proof}
\textbf{(Necessity)} It is well known that the exceptional set of points in $T$ such that $f$ fails to have radial
 limits is of type $G_{\delta\sigma}$ and by {Fatou's Theorem} is of measure zero. 
 Also if $f$ has unrestricted limits on $T\setminus E_2$ then $E_2=\bigcup_{n=1}^{\infty} F_n$ where $F_n$ are the closed sets consisting 
 of all points in $E_2$ any neighborhood of which contains points $t_1,t_2\in D$ such that $|f(t_1)-f(t_2)|>\frac{1}{n}$. Since clearly each $F_n$ is closed it follows that $E_2$ is of type $F_{\sigma}$ and the proof of necessity is complete.\\
\textbf{(Sufficiency)} Since $E_1$ is of type $G_{\delta\sigma}$ and of measure zero by {Kolesnikov's Theorem} there exists a bounded analytic function $f_1$ that fails to have radial limits exactly on $E_1$. By adding a suitable constant we may assume that $f_1$ does not have vanishing radial limits anywhere on $T$.\\
 Now since $E_2$ is $F_\sigma$ it follows $T\setminus E_2$ is of type $G_\delta$ and therefore by {Theorem 2.4.3} there exists a non zero bounded analytic function $f_2$ such that $f_2=0$ exactly on $T\setminus E_2$ and such that the radial limits of $f_2$ exist at all points of $T$.
The function $F\equiv f_1f_2$ satisfies the conditions of the theorem hence the proof is complete. 
\end{proof}

 Now, before we move on to solve another more complex case, let us recall some fundamental results due to Rene-Louis Baire.\\
 Recall that \emph{Baire's Theorem} (See \cite{S}, p. 986) asserts that every function $f:X\rightarrow Y$, taking values from a complete metric space $X$ to a Banach space $Y$ and of the \emph{Baire first class}; that is, an $f$ which is the pointwise limit of a sequence of continuous functions, has the following property: The points of continuity of $f$ are dense in $X$. \\
 From that, it follows that the points of discontinuity of $f$ are nowhere dense in $X$. This fact will also be used later to simplify a proof of a theorem found in \cite{PE}.

\begin{thm} Let $E_1\subset E_2$ be subsets of the unit circle $T\equiv\lbrace z\mid |z|=1 \rbrace$. Moreover assume that $E_1$ is closed and that $E_1$ and $E_2\setminus E_1$ are separated. The following conditions are necessary and sufficient for the existence of a Blaschke product $B(z)$ such that the radial limits of $B(z)$ fail to exist precisely at $E_1$ and it has unrestricted limits precisely at $T\setminus E_2$:\begin{enumerate}
\item $E_1$ is of measure zero.
\item $E_2$ is closed and nowhere dense.
\end{enumerate}
\end{thm}

\begin{proof}
\textbf{(Necessity.)} The proof that $E_1$ is of measure zero follows from {Fatou's Theorem} and that $E_2$ is closed follows since $E_2$ is clearly the accumulation of zeros of $B(z)$. Moreover, since $E_2$ and $E_1$ are closed and separated by hypothesis, it follows that $E_2\setminus E_1$ is closed as well. Now to prove that $E_2$ is nowhere dense, we consider the closed sets $T\setminus O_n$, where $$O_n\equiv \bigcup_{e^{i\theta}\in E_1}(e^{i(\theta-\frac{1}{n})},e^{i(\theta+\frac{1}{n})}).$$  Note that for all $n$, the radial function $B(e^{i\theta})$ is of the Baire first class on the closed set $T\setminus O_n$. The points of discontinuity of our function $B(e^{i\theta})$ restricted on $T\setminus O_n$ is the set $E_2\setminus O_n$. It follows by Baire's theorem that the set $E_2\setminus O_n$ is nowhere dense on $T\setminus O_n$. Since $E_2\setminus O_n$ is closed, it follows that it is actually nowhere dense in the standard topology on $T$ as well. Therefore the closed set $E_2\setminus E_1=\bigcup_{n=1}^\infty E_2\setminus O_n$ is nowhere dense as well. Lastly, since $E_1$ is a closed set of measure zero, of course it is nowhere dense. Now notice that $E_2= (E_2\setminus E_1)\cup E_1$, hence it follows that $E_2$ is nowhere dense and the proof of necessity is complete.\\
\textbf{(Sufficiency.)} Since $E_1$ is a closed set we know by 
{Theorem 2.1.9} there exists a Blaschke product $B_1(z,A_1)$ whose radial limits fail to exist exactly on ${A_1}^{'}=E_1$ and it's 
unrestricted limits exist everywhere on $T\setminus E_1$. Now since $E_2\setminus E_1$ is closed and nowhere dense, by {Theorem 3.4.1} (\emph{Colwell's Theorem} which is introduced in the next Chapter) there exist a Blaschke product $B_2(z,A_2)$ with ${A_2}^{'}=E_2\setminus E_1$ and such that $B_2(z,A_2)$ has radial limits everywhere on $T$ and has unrestricted limits exactly only on $T\setminus (E_2\setminus E_1)$. Then the Blaschke product $B(z,A_1\cup A_2)$ satisfies the conditions of the theorem and the proof is complete.
\end{proof}

 Now we note another special case of the Unification problem that can also be solved for Blaschke products provided we assume that the set $E_2$ is of measure zero. To prove this special case we use the following theorem found in \cite{NI}.

\begin{thm} Let $E$ be a subset of the unit circle of zero Lebesgue measure and of type $F_\sigma$ and $G_\delta$. Let $\phi$ be a function defined on $E$ with $\sup\lbrace |\phi(e^{it})|:e^{it}\in E\rbrace\leq 1$ and such that for each open set $U$ in the complex plane, $\phi^{-1}(U)$ is of type $F_\sigma$ and $G_\delta$. Then there exists a Blaschke product $I$ extending analytically on $T\setminus \overline{E}$ such that $$\lim_{r\rightarrow 1} I(re^{it})=\phi(e^{it})\text{ for }e^{it}\in E.$$
\end{thm}

\noindent\textbf{Remark 2.4.1.}
We know by the Factoraization theorem that every bounded analytic function $f$ can be decomposed in the product of three specific bounded analytic functions. In particular we know $f=BSF$ where $B$ is a Blaschke product $S$ a singular inner function and $F$ an outer function. The idea is that by studying the Unification Problem under the lens of each of the specific factors in the factoraization theorem one would hope to eventually completely solve the Unification Problem. In this Thesis we will deal only with Blaschke products, and we leave the study of singular inner functions and outer functions for a future research, which may eventually completely solve the Unification Problem.\\

 Now we introduce and solve the following special case of the Unification Problem regarding Blaschke products.
\begin{thm} Let $E_1\subset E_2$ be subsets of the unit circle $T\equiv\lbrace z\mid |z|=1 \rbrace$. Moreover, assume that $E_2$ is of measure zero. Then there exists a Blaschke product $B$ such that the radial limits of $B$ fail to exist precisely at $E_1$ and it has unrestricted limits precisely at $T\setminus E_2$ if and only if $E_2$ is closed and $E_1$ is of type $G_{\delta\sigma}$ and of measure zero.
\end{thm}

\begin{proof}\textbf{(Necessity.)}
The proof that $E_2$ is closed follows since any Blaschke product has unrestricted limits only on the set complimentary to the accumulation of its zeros. This implies that the set $E_2$ is the closed set in which the zeros of our Blaschke product accumulate. The fact that $E_1$ is $G_{\delta\sigma}$ is well known and as noted earlier its proof may be found in \cite{AL}. \\
\textbf{(Sufficiency.)}
Since $E_1$ is a $G_{\delta\sigma}$ of measure zero we know by {Kolesnikov's Theorem} that there exists a bounded analytic function $f_1$ with no radial limits exactly on $ E_1$. Moreover, since $E_2$ is a closed zero measure set we can apply {Theorem 2.4.8} to find a Blaschke product $B_2$ such that the radial limits of $B_2$ are equal to $\frac{1}{2}$ for all $e^{it}\in E_2$. Moreover, since $E_2$ is closed the conditions of Nicolau's Theorem are satisfied as well. It follows that for every bounded analytic function $g$ there exists a Blaschke product $I$ mimicking its boundary behavior on the set $E_2$. Now the product function $g(z)\equiv f_1(z)\cdot B_2(z)$ has the desired boundary behavior restricted on $E_2$; that is $g$ has no radial limits exactly on $E_1$ and has radial but not unrestricted on $E_2\setminus E_1$ . Thus, by Nicolau's Theorem we can find a Blaschke product $I$ that extends analytically on $T\setminus E_2$ and such that $$\lim_{r\rightarrow 1}[ I(re^{it})-g(re^{it}]=0$$
Then the Blaschke product $I$ has the desired boundary behavior and the proof is complete.
\end{proof}

\noindent\textbf{Remark 2.4.2.}
\noindent Recall that {Kolesnikov's Theorem} states that for any set $E$ of type $G_{\delta\sigma}$ on the unit circle there exists a bounded analytic function which fails to have radial limits precisely on that set. Thus the Characteraization Theorem above can be thought of as an analogue of {Kolesnikov's Theorem} for Blaschke products, but with the extra requirement that $\overline{E}$ is of measure zero. 

\chapter{Frostman's condition and Collwel's Theorems}
\section{Application of Frostman's condition in boundary behavior problems}
 \indent \indent Frostman proved the following well known condition regarding the existence of angular limits of a Blaschke product on the boundary.  
 
\begin{thm}[Frostman's condition] Let $A$ be a sequence of points inside the unit disc $D$ and let $B(z,A)$ denote the Blaschke product whose zero set is $A$. Then a necessary and sufficient condition that $\lim_{r\rightarrow 1}B(re^{i\theta^0})=L$ where $|L|=1$ is that \begin{equation}
\sum_A\frac{1-|a|}{|e^{i\theta}-a|}<\infty.
\end{equation}
\end{thm}

\noindent\textbf{Definition 3.1.1.} We define $B_L$ as the subclass of $H^\infty$ consisting of Blaschke products $B(re^{i\theta})$ such that for all $\theta \in [0,2\pi]$ $\lim_{r\rightarrow 1}B(re^{i\theta})=L(\theta)$ exists.\vspace{0.12in}

 Now we will utilize {Frostman's condition} $(3.1)$ and prove the following:

\begin{prp} Let $B(z,A)\in B_L$ be a Blaschke product with zero set $A=\lbrace a_n \rbrace$, then the subset of $A'$ in the unit circle in which $B(z,A)$ has radial limits of modulus $1$ is a set that is both $G_{\delta}$ and $F_\sigma$. \end{prp}

\begin{proof} Indeed if we define the following sequence of functions $\lbrace f_n \rbrace$

\begin{align*}
f_n:[0,2&\pi]\rightarrow\mathbb{R}_{>0}\\
&\theta\longmapsto\sum_{k=1}^{n}\frac{1-|a_k|}{|e^{i\theta}-a_k|}.
\end{align*}

 We note that the sequence $\lbrace f_n\rbrace$ is a sequence of continuous functions on $[0,2\pi]$ and the sequence is unbounded exactly at the values of $\theta$ where {Frostman's condition} $(3.1)$ from above is not satisfied. That is exactly at the set of points where the Blaschke product $B(z,A)$ fails to have radial limits of modulus $1$. Now, by a well known theorem the set of points where the sequence $\lbrace f_n\rbrace$ is unbounded is of type $G_\delta$. It follows that the set of points where the radial limit exists and is of modulus $1$ is of type $F_\sigma$. Now let $r_n=1-\frac{1}{n}$ and consider the following set of continuous functions $$B^*_n(r_ne^{i\theta})=\frac{1}{1-|B_{{r_n}}(e^{i\theta})|}.$$

 Then, clearly the set of points $\lbrace e^{i\theta}\rbrace$ for $\theta \in [0,2\pi]$ where the sequence of functions $B^*_n$ is unbounded is exactly the set of point where $\limsup_{r\rightarrow 1}B(re^{i\theta})=1$. But since out Blaschke product $B(z)$ is in the class $B_L$, this is exactly the set of points where $\lim_{r\rightarrow 1}B(re^{i\theta})=1$. Therefore by the well known theorem this set is of type $G_\delta$ and the proof is complete.
\end{proof}

\section{Collwel's Theorems}
\indent \indent The following theorems are from Colwell's paper \cite{PE}; On the boundary behavior of radial limits of Blaschke products. The original proof of the necessity is long and complex. In this section, we will present a straightforward method of proving the result based on Baire's theorem, or even an elementary theorem. \\
\indent Let $A'$ be the set of limit points of $A$ on $T$. The following theorems are the main results of P. Colwell's paper \cite{PE} (see Theorem 2 and Theorem 3 in \cite{PE}).
\begin{thm} Let $E\subset T$. There exists a Blaschke product $B(z;A)$ for which $B(e^{i\theta})$ is defined and of modulus one at every point of $T$ and $A'=E$ where $A$ is the set of zeros of our Blaschke product $B(z;A)$ if and only if $E$ is closed and nowhere dense.
\end{thm}

\begin{thm}
Let $B(z)$ be a Blaschke product with $B(e^{i\theta})$ defined and of modulus 1 at every point of $T$. Then, as a function $B(e^{i\theta})$ is discontinuous at $\theta=\theta_0$, if and only if $e^{i\theta_0}$ is in $A'$.
\end{thm}

 Our purpose now is to simplify the proofs of {Theorem 3.2.1} and {Theorem 3.2.2} of \cite{PE}. In fact we will show that the necessity parts of both theorems follow directly by two classical theorems: the elementary theorem on the boundary continuity of the Poisson integral and {Baire's Theorem} on the first class functions. In fact, we only need a particular case of the necessity part of {Baire's Theorem} or just a well known elementary argument.\\
 \indent Such proofs have remained unnoticed for a long time even though \cite{PE} has been discussed by known authors (including Colwell, Noshiro and Cargo) and also \cite{PE} has been cited in many further works in Blaschke products. However, it should be noted that except Baire's and Fatou's theorems, our simple approach uses {Theorem 2.1.3} (Riesz Theorem).\\
As Cargo pointed \cite{CA}, both {Theorem 3.2.1} and {Theorem 3.2.2} remain valid if the phrase "and of modulus one" is omitted. The appropriate statements are:

\begin{thm} Let $E\subset T$. There exists a Blaschke product $B(z;A)$ for which $B(e^{i\theta})$ is defined at every point of $T$ and $A'=E$ where $A$ is the set of zeros of our Blaschke product $B(z;A)$ if and only if $E$ is closed and nowhere dense.
\end{thm}

\begin{thm}
Let $B(z)$ be a Blaschke product with $B(e^{i\theta})$ defined at every point of $T$. Then, as a function $B(e^{i\theta})$ is discontinuous at $\theta=\theta_0$, if and only if $e^{i\theta_0}$ is in $A'$.
\end{thm}
 Our new method for the proofs is equally well applicable for the necessity parts of these results.
Note that one of the proofs of the necessity part of {Theorem 3.2.1} given in \cite{PE} is not valid in case of {Theorem 3.2.3}. We will discuss this below in details. We also offer some clarifications on the proof of the necessity part of {Theorem 3.2.1} in \cite{PE}.

\section{Auxilliary results}

\indent \indent The following simple theorem of continuity of the Poisson integral at a boundary point is due to Fatou (see \cite{GO} p. 381, Theorem 1 or \cite{AZ}, p. 98, a special case of Theorem 6.13).

\begin{thm} Let $f$ be a periodic and integrable function on $[-\pi, \pi]$, and let $U(z) (z\in D)$ be the Poisson integral of $f$. If $f$ is continuous at a point $\theta=\theta_0$, then $U(z)$ tends to $f(\theta_0)$ as $z=re^{i\theta}$ approaches the point $e^{i\theta_0}$ arbitrarily in the open unit disc $D$.
\end{thm}

\begin{proof}
Let $A=f(e^{i\theta_0})$, then we need to show that for any $\epsilon>0$ there exists $\delta_\epsilon$ and $r_\epsilon$ such that$\text{ for all } r>r_\epsilon$ $$|re^{i\theta}-re^{i\theta_0}|<\delta_\epsilon\Rightarrow |U(re^{i\theta})-A|<\epsilon.$$
To this end we fix $\epsilon>0$, then for any $\delta>0$ \begin{align*}
|U(z)-A|&=\frac{1}{2\pi}\vert\int_{-\pi}^{\pi}(f(e^{i\theta})-A)p(r,\theta-t)dt\mid\\
&=\frac{1}{2\pi}\mid\int_{-\delta}^{\delta}(f(e^{i\theta})-A)p(r,\theta-t)dt+\int_{|t|>\delta}(f(e^{i\theta})-A)p(r,\theta-t)dt\mid\\
&\leq\frac{1}{2\pi}\int_{-\delta}^{\delta}\mid(f(e^{i\theta})-A)p(r,\theta-t)\mid dt+\int_{|t|>\delta}\mid(f(e^{i\theta})-A)p(r,\theta-t)\mid dt.
\end{align*}
 Now, since $f$ is continuous at $e^{i\theta_0}$, there exists $\delta_\epsilon$, such that $$|e^{i\theta}-e^{i\theta_0}|<\delta_\epsilon\Rightarrow|f(e^{i\theta})-A|<\frac{\epsilon}{2}$$ Moreover, since the Poisson kernel is an approximate identity, there exists $0<r_\epsilon<1$ such that $$r>r_\epsilon\Rightarrow\int_{|t|>\delta_\epsilon}\mid(f(e^{i\theta})-A)p(r,\theta-t)\mid dt<\frac{\epsilon}{2}$$
 Therefore, for $\delta<\delta_\epsilon$ and $r>r_\epsilon$ $$|U(z)-A|\leq\frac{1}{2\pi}\int_{-\delta_\epsilon}^{\delta\epsilon}\mid(f(e^{i\theta})-A)p(r,\theta-t)\mid dt+\int_{|t|>\delta_\epsilon}\mid(f(e^{i\theta})-A)p(r,\theta-t)\mid dt<\epsilon.$$
\noindent The proof is complete.
\end{proof}

 In {Theorem 3.3.1} the function $f$ is either a real valued or complex valued function. Obviously, the complex valued case follows from the real valued case by applying the theorem to the real and imaginary parts of $f$. The short and elementary proof of {Theorem 3.3.1} merely uses the basic properties of the Poisson kernel discussed in Chapter $1$.\vspace{0.12in}

 The next theorem is a corollary of the necessity part of R. Baire's topological theorem on the first class functions (see \cite{BA}, p. 462, {Baire's Theorem}, or \cite{HA}, p. 287, Theorem V), which we formulate for functions defined on $T$.

\begin{thm}
Let $f$, $f_n$ be defined on $T$ and let $f_n$ be continuous on $T$. If $\lbrace f_n\rbrace$ converges to $f$ at each point of $T$, then $f$ is continuous on a dense subset of $T$.
\end{thm}

 In the next section we will give a proof of {Theorem 3.3.2} based on a well known elementary theorem.

\section{Proofs}

\begin{proof}[Proof of necessity of {Theorem 3.2.1}.] Suppose that $B(z)$ has radial limits of modulus one everywhere on $T$ and that $E=A'$. Then $E$ is closed. Therefore by {Theorem 3.3.1} and {Theorem 3.3.2} it follows that on a dense subset $M$ of $T$ the function $B(z)$ has continuous boundary values. {Riesz Theorem} implies that each continuous boundary value is of modulus one. Since $A$ is the set of zeros of $B(z)$, no point of $M$ can be in $A'=E$, and the closed set $E$ must be nowhere dense on $T$.
\end{proof}

\begin{proof}[Proof of necessity of {Theorem 3.2.3}.] Suppose that $B(z)$ has radial limits everywhere on $T$ and that $E=A'$. The rest of the proof is identical to the proof above.
\end{proof}

\begin{proof}[Proof of necessity of {Theorem 3.2.2}.]
If $e^{i\theta_0}$ is not in $A'$ then it is well known that $B(z)$ is even analytic in some neighborhood of $e^{i\theta_0}$. Thus if $e^{i\theta_0}$ is discontinuous at $\theta=\theta_0$, then $e^{i\theta_0}\in A'$.\\
 Now let $e^{i\theta_0}$ be in $A'$. We will show that $B(e^{i\theta})$ is discontinuous at $\theta=\theta_0$. Assume for a contradiction that $B(e^{i\theta})$ is continuous at $\theta=\theta_0$. Then, by {Theorem 3.3.1}, $\lim_{z\rightarrow e^{i\theta_0}}B(z)=B(e^{i\theta_0})$ (where $z\rightarrow e^{i\theta_0}$ arbitrarily in $D$). Therefore, by Riesz Theorem we must have that $|B(e^{i\theta_0}|=1$. But since $e^{i\theta_0}\in A'$ we must also have that $B(e^{i\theta_0})=0$, i.e a contradiction. The proof is complete.
\end{proof}
\begin{proof}[Proof of {Theorem 3.2.4}.] This proof is identical to the proof of {Theorem 3.2.2} above.
\end{proof}

\begin{proof}[Proof of {Theorem 3.3.2}.] (We follow the classical approach but apply simplifications that conveniently fit our particular case). We will show that for any closed arc $[a,b]\subset T$ the function $f$ has a continuity point in $(a,b)$. Let $\epsilon>0$ be given. The set $F_{n,m}=\lbrace t\in [a,b]: |f_n(t)-f_{n+m}(t)|\leq\epsilon\rbrace$ is closed and $G_n=\bigcap_{m=1}^\infty F_{n,m}$ is closed. Since $f_n$ converges to $f$ on $[a,b]$, we have $[a,b]=\bigcup_{n=1}^\infty G_n$.\\
 Since $[a,b]$ is not of first category on itself (by a theorem of Baire), there exists $n_0$ and an open interval $I$ $(I\subset T)$ containing points of $[a,b]$ such that $[a,b]\cap I$ contains a point $t_0$ different than $a$ and $b$. Let $J\subset T$ be a closed arc containing $t_0$ in its interior and so small that:
 \begin{enumerate}
\item[i.] $J\subset I$.
\item[ii.] $J\subset (a,b)$.
\item[iii.] The oscillation of the continuous function $f_{n_0}$ is less than $\epsilon$.
\end{enumerate}
 Since also $|f_{n_0}(t)-f(t)|\leq\epsilon$ on $J$, by the triangle inequality it follows the oscillation of $f$ on $J$ is less than $3\epsilon$. This implies that there exists closed arcs $J_n=[c_n,d_n]\subset (a,b)$ such that \begin{enumerate}
\item[(a)] $J_{n+1}=(c_{n+1},d_{n+1})\subset (c_n,d_n)$.
\item[(b)] The length of $J_n$ is less than $1/n$.
\item[(c)] The oscillation of $J_{n+1}$ is less than $1/n$.
\end{enumerate}
Then $f$ is continuous at a point $\tau\in\bigcap_{n=1}^\infty J_n$ and also $\tau\in (a,b)$ and the proof is complete.
\end{proof}

 The following theorem is a simplification of a proof of a theorem by A.J Lohwater and G.Piranian found in \cite{PI}. The proof follows the methodology used to prove {Theorem 2.4.5} while also utilizing the sufficiency part of {Theorem 3.3.1} proven in \cite{PE}.

\begin{thm}[Lohwater and Piranian] Let $K$ be an $F_\sigma$ set of first category in the unit circle $T$. Then there exists a bounded analytic function $\Phi(z)\in H^\infty$ such that the radial limits of $\Phi(z)$ exists everywhere on $T$ and has unrestricted limits exactly on $T\setminus K$.
\end{thm}
\begin{proof}
Since $K$ is $F_\sigma$ and of first category, it follows $K=\bigcup_{i=1}^\infty K_i$ where $K_i$ are closed nowhere dense subsets of $C$. Then for each $i$ by the theorem in Colwell's paper mentioned above there exists a Blaschke product $B(z,A_i)$ that has radial limits of modulus $1$ everywhere on $T$ and has unrestricted limits of modulus $1$ exactly on $T\setminus K_i$.\\
Now since for all $z\in D$ $\mid B(z,A_i)\mid<1$, it follows that $\sum_{i=1}^\infty \frac{\mid B(z,A_i)\mid}{i^2}<\frac{\pi^2}{6}$, thus the infinite sum $\sum_{i=1}^\infty(\frac{B(z,A_i)}{i^2})$ converges uniformly to an analytic function $\Phi(z)$ on $D$. Now since each $B(z,A_i)$ has radial limits everywhere on $T$ and unrestricted limits of modulus $1$ exactly on $T\setminus K_i$, it follows that $\Phi(z)$ has radial limits everywhere on $T$ and unrestricted limits exactly on the $G_\delta$ set $T\setminus K$. The proof is complete.
\end{proof}

\section{Further discussions}
\indent \indent In this section our main purpose is to show that the method of the proof of the necessity of {Theorem 3.2.1} (Theorem 2 in \cite{PE}) does not apply for the proof of the necessity part of {Theorem 3.2.3}.\\ \indent For the proof of the necessity of {Theorem 3.2.1} in $\cite{PE}$ the author constructs a polygonal path $P(z)$ inside of $D$ ending at a point of $T$. Taking the limit of $B(z)$ along $P(z)$ the author uses the expression "if it exists"(referring to the limit along the path). However, that limit certainly exists by {Lindel\"{o}f's Theorem}, to which the author \cite{PE} gives a reference.\\
 The final part of the proof of the necessity of Theorem 2 on \cite{PE} perhaps can be slightly modified as follows. Since the radial limits of $B(z)$ exists (at all points of $T$), {Lindel\"{o}f's Theorem} implies that also angular (or non tangential) limits exist. Since the polygonal path $P(z)$ approaches to a point on $T$ non tangentially, the limit of $B(z)$ along $P(z)$ exists. This limit is not zero as it must be of modulus one by assumption. But this is impossible since the zeros of $B(z)$ are located on $P(z)$. This contradiction completes the proof.\\
Since {Theorem 3.2.3} allows radial limits (including zero) one cannot arrive to a similar contradiction for {Theorem 3.2.3}.

\chapter{Arakeljan's Theorem}

\section{Introduction}

\indent In this chapter we will generalize an important theorem in the field of Complex approximation, namely we will generalize \emph{Arakeljan's Theorem}. It is well known by Weirstrass Approximation Theorem that in a real closed interval, any continuous function can be approximated uniformly by polynomials. In the complex case things are different. Since polynomials are analytic, by Morera's Theorem it follows that if $f$ is uniformly approximated by polynomials, then $f$ better be an analytic function. However, $f$ being analytic is not enough to guarantee uniform approximation by polynomials. This is indeed the case if we let $f(z)={1}/{z}$ defined on the annulus $F=\lbrace z\in\mathbb{C}:\frac{1}{2}<z<2\rbrace.$ Even though $f$ is analytic in $F$, in this case $f$ cannot be approximated by polynomials because by Cauchy's theorem if $f$ was approximated by polynomials, the contour integral around the unit circle would have to be equal to zero. However direct computation shows it is equal to $\frac{1}{2\pi}$ instead. Of course this happens because $f$ cannot be analytically continued inside the unit circle, which occurs because $F$ has a bounded connected component(a hole) in its complement.\\
\indent  A special case of \emph{Runge's Theorem} states that any function $f$ that is analytic on a domain $G$ which contains a compact set $F$, can be uniformly approximated by polynomials provided $F$ has a connected complement. Runge's Theorem is a classical theorem and it essentially serves as the starting point of Complex Approximation. The theorem in all of its generality does not require $F$ to have a connected complement, and thus the approximation in general is not by polynomials but rather by rational functions. The theorem also asserts that the poles of those rational functions may be chosen freely to lie anywhere in the respective connected component on $G\setminus F$ they lie at. As a corollary when $G\setminus F$ has a single component the pole may be chosen to be at "$\infty$", thus $f$ may be approximated by polynomials rather than rational functions. Another well known theorem is \emph{Mergelyan's Theorem} which improves Runge's Theorem when $F$ has a connected complement. The theorem states that any function that is continuous on a compact set $F\subset\mathbb{C}$ and analytic on $F^o$ can be uniformly approximated by polynomials on $F$, provided our set $F$ has a connected complement.\\
\indent {Arakaljan's Theorem} on the other hand deals with analytic functions defined in more general sets $F$ that are neither compact nor do they have connected complement. Consequently we don't expect the approximation to be by polynomials. \\
To this end, let $G\subset\mathbb{C}$ be an arbitrary domain and $F\subset G$ a relatively closed subset of $G$. {Arakeljan's Theorem} provides conditions that determine whether a function $f\in A(F)$ can be approximated on $F$ by functions $g\in Hol(G).$ 
\begin{align*}
Hol(G)=\lbrace  g: g\text{ analytic in }G\rbrace \\
 A(F)=\lbrace  f: f\text{ analytic on }F^o,\text{ continuous on F}\rbrace
\end{align*}
\indent The main result in \cite{RU} generalizes Arakeljan's Theorem when $G=\mathbb{C}$. In \cite{RU} the authors consider two closed sets $F$ and $C$ and find necessary and sufficient conditions so that every function defined on $F\subset \mathbb{C}$ can be approximated by entire functions that are bounded on $C$. In this paper we will extend this even further by considering the result for a more general $G$. We will see that if our sets $C$ and $F$ satisfy a certain condition in $G$, or if $G$ is a simply connected then Arakeljan's theorem can be extended.\\
\indent The results in this paper fit in the category of joint approximation. We provide some relevant articles where similar problems have been considered. In particular, articles (\cite{RA}-\cite{AD}) are devoted to various problems regarding bounded approximation by polynomials.
Whereas articles (\cite{Le}-\cite{MU}) are devoted to problems regarding the joint approximation by analytic functions in Banach Spaces.

\section*{Description of Arakelan sets}
 In the proof of {Arakeljan's Theorem} in (\cite{GA}, p. 142-144) the definition used for $F$ to be an Arakeljan set is different than definition we will use in this paper. The definition used in (\cite{GA}, p. 142) is more topological but it turns out that both definitions are equivalent. The definition we will use in this paper is also used in \cite{RU} to prove an extension of Arakeljan's Theorem for the case $G=\mathbb{C}$.\\
The definition we will use for an Arakeljan set in this paper is based on the concept of a $G-$hole. 

\noindent\textbf{Definition 1.}
Let $G\subset \mathbb{C}$ be an arbitrary domain and $F\subset G$ be a relatively closed subset of $G$. A non empty connected component $g$ of $G\setminus F$ is called a $G-$hole of $F$ if it can be enclosed in a compact subset $L\subset G.$
\\

\noindent\textbf{Remark 1.} It is important to note that a $G-$ hole of $F$ is essentially a bounded connected component of the complement (a hole) whose boundary does not intersect with the boundary of $G$.\vspace{0.12in}

\noindent Now that we defined what a hole is in the setting of Arakeljan's theorem on an arbitrary domain, we are ready to define Arakeljan sets.\vspace{0.12in}

\noindent\textbf{Definition 2.} Let $G$ be an arbitrary domain and let $F$ be a relatively closed subset, then we call $F$ an Arakeljan set if: \begin{enumerate}
\item The set $F$ has no $G-$holes.
\item For any connected compact set $K\subset G$ such that $\partial K$ is the union of disjoint jordan curves\footnote{So that locally, the neighborhood around any point on $\partial K$ looks like an open interval.}, the set $H\equiv\lbrace \bigcup \lbrace h\rbrace :h \text{ is a $G$-hole of }F\cup K\rbrace$ can be contained in a compact subset of $G$.
\end{enumerate}

\noindent\textbf{Remark 2.} Clearly since the boundary of $G-$holes does not intersect with the boundary of $G$, it follows that if $H$ has only a finite number of $G-$holes and $F$ has no $G-$holes, then $F$ is an Arakeljan set. Additionally, it follows that if $F$ has no $G-$holes and it fails to be an Arakeljan set, then either $H$ must be unbounded or $\partial H\cap \partial G\not=\emptyset$.\vspace{0.12in}\\
Now, in order to compare our definition with the more topologically flavored definition presented in (\cite{GA}, p. 142), we will need to some tools from topology. To this end let us recall the \emph{Alexandroff compactification} of an arbitrary domain $G\subset \mathbb{C}$. The set $G^*$ is defined by introducing the point $\infty$ so that $G^*=G\cup \lbrace\infty\rbrace$. The topology of $G^*$ is defined to consist from the open sets of $G$ and in addition all the sets that are complements of compact subsets $K\subset G$. Under this topology one may check that $G^*$ is compact.\\
We are ready now to provide the definition of an Araklejan set as presented in (\cite{GA}, p. 142).\vspace{0.12in}
 \pagebreak
 
\noindent\textbf{Definition 3.} The conditions given on the set $F$ to be an Arakeljan set are: \begin{enumerate}
\item$G^*\setminus F$ is connected.
\item$G^*\setminus F$ is locally connected at $\infty$.
\end{enumerate}

\noindent The following proposition introduced in (\cite{GA}, p. 133) provides us with an intuitive approach for determining whether $G^*\setminus F$ is connected.

\begin{prp}
The space $G^*\setminus F$ is connected if and only if each component of the open set $G\setminus F$ has an accumulation point on $\partial G$ or is unbounded.
\end{prp}

\begin{proof}
The proof can be found in (\cite{GA}, p. 133-134).
\end{proof}

Even though the conditions of {Definition 3} given above seem different from the ones in {Definition 2}, actually they are the same.
\begin{prp}
Definition 2 and Definition 3 describing Arakeljan sets are equivalent.
\end{prp}  

\begin{proof}
Indeed, let $G$ be an arbitrary domain and $F$ a relatively closed subset of $G$.
By {Proposition 1}, the set $G^*\setminus F$ is connected if and only if each component of the open set $G\setminus F$ has an accumulation point on $\partial G$ or is unbounded. Therefore, $G\setminus F$ is connected if and if every hole (component of $G\setminus F$) is not a $G-$hole. Therefore, conditions 1 of both definitions are equivalent.
Now we will show that conditions 2 of both definitions are equivalent as well.\\
Indeed, suppose that for any connected compact set $K$ where $\partial K$ is the union of disjoint jordan curves, the set of $G-$holes of $(F\cup K)$ can be contained in a compact subset of $G$. We will show that this implies that $G^*\setminus F$ is locally connected at $\infty$.
Now, recall that $G^*\setminus F$ being locally connected at $\infty$ by definition requires that every open set of $G^*\setminus F$ containing $\infty$ contains a connected sub-neighborhood containing $\infty$. However, every open set in $G^*\setminus F$ is simply the 
complement of a compact set $K\subset G\setminus F$.\\
Suppose for a contradiction that $G^*\setminus F$ is not locally connected. Then there must exist a compact set $K\subset G\setminus F$ such that for all compact subsets $C\subset G$, where $K\subset C$, the set $G^*\setminus (C\cup F)$ is not connected. This of course can only happen when the set of $G-$holes of $C\cup F$ cannot be contained in a compact subset of $G$. Now since $G$ is open and path connected we can enclose each component of $K$ either by a compact set homeomorphic to a disk or by a compact set homeomorphic to a wedge of annuluses. Then since $G$ is path connected we can join the enclosed components by a tube\footnote{If $\gamma$ is a path joining two enclosed components of our compact set, say $A$ and $B$. If we remove a small interval from $\partial A$ and a small interval from $\partial B$, then since $G$ is open we can join $\partial A$ with $\partial B$ via a rectangular path "tube" that follows the path of $\gamma$} to form a set $L\supset K$, such that $\partial L$ is the union of disjoint jordan curves and the set of $G-$holes $F\cup L$ cannot be contained in a compact subset of $G$, i.e a contradiction.\\
The other direction is obvious since for any compact set $K\subset G$ that fails condition 2 of definition 2, the neighborhood $G^*\setminus(F\cup K)\subset G^*\setminus F$ is not locally connected at $\infty$. The proof is complete.
\end{proof}

\noindent Now we are ready to formally state Arakeljan's Theorem.

\begin{thm}[Arakeljan's Theorem]
Let $F$ and $G$ be as above, then any $f\in A(F)$ can be uniformly approximated by functions $g\in Hol(G)$ if and only if $F$ is an Arakeljan set.
\end{thm}

\begin{proof}
Arakeljan's Theorem was first stated and proved in \cite{AR} by an Armenian Mathematician named Arakeljan. An additional proof may be found in (\cite{GA}, p. 142-144).
\end{proof}

Our goal now is to move past the restriction $G=\mathbb{C}$ and prove the following:

\begin{thm}[Extension of Arakeljan's Theorem]
Let $G\subset \mathbb{C}$ be an arbitrary domain and suppose $F$ is a relatively closed subset of $G$. Moreover assume that $C\subset G$ is a closed subset in $G$\footnote{Here we mean that $C$ is a closed set in $\mathbb{C}$ that is also a subset of $G$. In particular it is a relatively closed subset whose boundary does not intersect with the boundary of $G$.} such that $C$ and $F$ are $G-$hole independent. In order that every function $f\in A(F)$ can be approximated uniformly on $F$ by functions $g\in Hol(G)$, each of which is bounded on $C$, it is necessary and sufficient that: \begin{enumerate}
\item $F$ is an Arakeljan set.
\item There exists a closed Arakeljan set $C_1\not=G$, so that $C\subset C_1$ and $F\cap C_1$ can be contained in a compact set in $G$.
\end{enumerate}
\end{thm}

 \section*{Main result}
 In \cite{RU} the specific case where $G=\mathbb{C}$ Arakeljan's Theorem has been extended to the following:

\begin{thm}[Extension of Arakeljan's Theorem when $G=\mathbb{C}$]
Let $F$ and $C$ be closed sets in the complex plane, with $C\not=\mathbb{C}$. In order that every function $f\in A(F)$ can be approximated uniformly on $F$ by entire functions, each of which is bounded on $C$, it is necessary and sufficient that: \begin{enumerate}
\item $F$ be an Arakeljan set.
\item There exists an Arakeljan set $C_1$, $C_1\not=\mathbb{C}$, so that $C\subset C_1$, and $F\cap C_1$ is a bounded set in $\mathbb{C}$.
\end{enumerate}
\end{thm}

 \noindent The proof of the extended version of Arakeljan's Theorem for $G=\mathbb{C}$ stated above is based on the following two lemmas. 

\begin{lem} Let $f$ be an entire function, and let $C$ a closed subset in $G$, $C\not=G$. In order that $f$ be bounded on $C$ it is necessary and sufficient that there exists a closed Arakeljan set $C_1\subset G$ so that $C\subset C_1$ and $f$ is bounded on $C_1$.
\end{lem}

\begin{lem}
The union of two disjoint Arakeljan sets $E$ and $F$ is again an Arakeljan set.
\end{lem}

 In order to prove the Extension of Arakeljan's Theorem we have to generalize {Lemma 1} and Lemma 2 for a general domain $G$. The reason we did not use the more intuitive definition for Arakeljan sets and $G-$holes (see {Remarks 1 and 2}) is so that we can adjust the proofs presented in \cite{RU} to fit the case of an arbitrary domain $G$.
 
 \begin{lem}
Let $G\subset\mathbb{C}$ be an arbitrary domain and let $f\in Hol(G)$. Suppose $C$ is a closed subset in $G$. In order that $f$ is bounded on $C$ it is necessary and sufficient that there exists a closed Arakeljan set $C_1\subset G$ so that $C\subset C_1$ and $f$ is bounded on $C_1$.
\end{lem}

\begin{proof}
The proof of sufficiency is obvious so we will only prove necessity. Therefore, lets assume that there exists an $M>0$ such that $|f(z)|<M$ for all $z\in C$, our goal is to show that there is an Arakeljan set $C_1$ containing $C$ such that $f$ is bounded on $C_1$.
 Now since $G$ is open and $C$ is a closed subset in $G$, by continuity of $f$ on $C$, it follows that for each $z\in C$ there exists an open disk $U(z)$ centered at $z$ of radius less than $\delta_z$, so that $|f(t)|<M+1$ for all $t\in U(z)$. Furthermore, we pick $\delta_z< dist(C,\partial G)$, so that: 
 \begin{align*}
 &\overline {\bigcup_{z\in C}U(z)}\cap \partial G=\emptyset. &(*)
 \end{align*}
  The open cover $\lbrace U(z)\rbrace$ , $z\in C$ of the closed set $C$ has a locally finite subcover which we denote by $\lbrace U_n\rbrace$, $n=1,2,3,....$ Hence we have \begin{enumerate}
\item $C\subset \bigcup_{n=1}^\infty U_n$
\item For any compact subset $K$ of $G$ only a finite number of disks $\lbrace U_n\rbrace$ intersect $K$.\\
Write $E=\bigcup_{n=1}^\infty \overline{U_n}$. By 2, $E$ is closed.
\item $|f(z)|<M+1,\   \ z\in E$
\end{enumerate}
Let $C_1$ be the intersection of $G$ with the union of $E$ and its $G-$holes (if any exist). $C_1$ is the union of a locally finite collection of closed sets, thus it is closed and by $(*)$ is a subset of $G$. Furthermore by construction $C_1$ does not have any $G-$holes.
From $(3)$, by the maximum principle, we have $$|f(z)|<M+1, \   \ z\in C_1$$
By $(1)$, we also have $C\subset C_1$. Hence the lemma will be proved when we show $C_1$ is an Arakeljan set. Recall that to show that $C_1$ is Arakeljan set we must show that: \begin{enumerate}
\item[a)] The set $C_1$ has no $G-$holes
\item[b)] For any connected compact set $K\subset G$ the set $H\equiv\lbrace h: h \text{ is a $G-$ hole of }C_1\cup K\rbrace$ can be contained in a compact set $L\subset G$.
\end{enumerate}
Condition $(a)$ is satisfied by construction of the set $C_1$, thus by {Remark 2} it suffices to prove that for any such connected compact set $K\subset G$ such that $\partial K$ is the union of disjoint jordan curves, $C_1\cup K$ has only a finite number of $G-$holes.\\
 To this end, let $K$ be such a compact subset of $G$. By $(2)$ in the sequence $\lbrace \overline{U_n}\rbrace$, there exists only a finite number of disks $\overline{U_{n_1}},\overline{U_{n_2}},....\overline{U_{n_k}}$, each of which intersects $K$. The set $(\partial K)\setminus\bigcup_{m=1}^{m=k}\overline{U_{n_m}}$ is the union of a finite number of disjoint open intervals $I_1,I_2,...,I_p$ on $\partial K$. For any $G-$hole $h$ of $K\cup C_1$, there exists at least one interval $I_k\ \ (1\leq k \leq p)$ so that $I_k\subset \partial h$. (Otherwise $h$ would be a $G-$hole of $C_1$, which is impossible, since $C_1$ is without $G-$holes).\\
 If $h_1,h_2$ are two distinct $G-$holes of $K\cup C_1$ and $I_{k_1}\subset\partial h_1\ \ (1\leq k_1 \leq p), I_{k_2}\subset\partial h_2\ \ (1\leq k_2 \leq p)$, then $k_1\not=k_2$, since otherwise we would have $h_1\cap h_2\not=\emptyset$ which is impossible since $h_1$ and $h_2$ are distinct $G-$holes.\\
 Consequently, the number of $G-$holes of $K\cup C_1$ cannot be more than the number of intervals $I_k$. Thus the number of $G-$holes of $C_1\cup K$ is finite, hence by {Remark 2} it follows that $C_1$ is an Arakeljan set and the proof is complete.
\end{proof}

\noindent In order to state the corresponding analogue of Lemma 2 for the case $G=\mathbb{C}$, we first need to introduce the following definition.\\
\noindent\textbf{Definition 4.} Let $G\subset\mathbb{C}$ be a domain and let $F$ be a relatively closed subset of $G$. Then we call a \emph{strict hole} of $F$ any connected component of $G\setminus F$ (hole) that cannot be contained in a compact subset of $G$ (i.e a hole that is not a $G-$hole).\\
  \noindent \textbf{Definition 5.} Let $E$ and $F$ be disjoint relatively closed subsets of $G$. Then we say that $E$ and $F$ are $G-$hole independent if any non empty pairwise intersection of strict holes of $E$ with strict holes of $F$ is not a $G-$hole. 
\begin{lem}
Let $G$ be an arbitrary open domain in $\mathbb{C}$ and suppose $E$ and $F$ are two disjoint Arakeljan sets in $G$ that are $G-$hole independent. Then the set $E\cup F$ is an Arakeljan set in $G$.
\end{lem}

\begin{proof}  The set of holes of $E\cup F$ is the following: \begin{align*} 
  G\setminus (E\cup F)&=(G\setminus E) \cup (G\setminus F)\\
            &=\bigcup_{i=1}^\infty E_i\cap \bigcup_{j=1}^\infty F_j\\
            &=\bigcup_{i=1}^\infty G_i.
\end{align*}
Where $E_i$ are the set of holes of $E$, $F_j$ the set of holes of $F$ and $G_i=\bigcup_{j=1}^\infty (E_i\cap F_j).$ Since $E$ and $F$ are Arakeljan sets it follows that none of them has $G-$holes, and since they are $G-$hole independent it follows that indeed $E\cup F$ has no $G-$holes as required. Therefore, it remains to verify the latter condition of definition 2 describing Arakeljan sets. To this end, assume for a contradiction that there exists a connected compact subset $K\subset G$ where $\partial K$ is the union of disjoint jordan curves, and such that $H$, the set of all $G-$ holes of $E\cup F\cup K$ is either unbounded or $\partial H\cap \partial G\not=\emptyset$. Then by {Remark 2} it follows the set $E\cup F\cup K$ has infinite $G-$holes. Consequently, the set $H$ must have infinitely many components i.e $H=\bigcup_{i=1}^\infty \lbrace h_i\rbrace$ where each $h_i$ represents a $G-$hole of $E\cup F\cup K$. Now for fixed $i$, let $a_k^i,\ \ k=1,2,...$ be all the $G-$holes of $\overline{h_i}\cup K$. The connected compact set $\overline{h_i}\cup\bigcup_k \overline{ a_k^i} $ has no $G-$holes. Denote by $d_i$ the interior of this connected compact set which contains $h_i$. Now $\partial d_i$ consists of an open arc on $\partial K$\footnote{Since $\partial K$ is the union of disjoint jordan curves and $K$ is connected} and a connected compact set which we denote by $K_i$. That is, $K_i=\overline{(\partial d_i)\setminus\partial K}.$ Clearly, $K_i\subset E\cup F$. Hence $K_i$ lies completely either on $E$ or in $F$. Let $i_n$,\ \ $n=1,2,...,$ be all the natural numbers for which $K_{i_n}\subset E$ and let $i_l,\ \ l=1,2,...,$ be the remaining natural numbers such that $K_{i_l}\subset F$. Now by our assumption the set $H=\bigcup_{i=1}^\infty \lbrace h_i\rbrace$ is either unbounded or $\partial H\cap \partial G\not=\emptyset$ or both. The case where $H$ is unbounded is already dealt with on \cite{RU}. For the case $\partial H\cap \partial G\not=\emptyset$ we simply note that since $H$ has an accumulation point on $\partial G$, it follows that one or both of the sets $\lbrace h_{i_n}\rbrace $ or $\lbrace h_{i_l}\rbrace$ has an accumulation point on $\partial G$. Without loss of generality let's assume that $\lbrace h_{i_n}\rbrace $ has an accumulation point on $\partial G$. Now let us consider the holes of $\lbrace h_{i_n}\rbrace \cup K$. Since $K_{i_n}\subset E$ and $h_{i_n}$ is a $G-$hole of $E\cup F\cup K$, we see that there exists a $G-$hole $V_{i_n}$ of $E\cup K$ so that $h_{i_n}\subset V_{i_n}\subset d_{i_n}$. Hence the set $\bigcup_{n=1}^\infty\lbrace V_{i_n}\rbrace$ has an accumulation point on $\partial G$, which implies that the set $E$ is not an Arakeljan set. This is a contradiction hence the lemma is proved.
\end{proof}

\noindent \textbf{Remark 3.} Note that in the lemma above we must necessarily require $E$ and $F$ to be $G-$hole independent. Indeed, if we let $G$ to be the punctured unit disc and $E$ and $F$ circles with radii $\frac{1}{2}$ and $\frac{1}{4}$ respectively. Then we can easily verify that $E$ and $F$ are disjoint Arakeljan sets whose union $E\cup F$ has the annuls $A=\lbrace z\in\mathbb{C}: \frac{1}{4}<|z|<\frac{1}{2}\rbrace$ as a $G-$hole. This happens of course because they are not $G-$hole independent.\\

\noindent Now that we have the two lemmas in our arsenal we are ready to prove the Extension of Arakeljan's Theorem.

\begin{thm}[Extension of Arakeljan's Theorem]
Let $G\subset \mathbb{C}$ be an arbitrary domain and suppose $F$ is a relatively closed subset of $G$. Moreover assume that $C\subset G$ is a closed subset in $G$ such that $C$ and $F$ are $G-$hole independent. In order that every function $f\in A(F)$ can be approximated uniformly on $F$ by functions $g\in Hol(G)$, each of which is bounded on $C$, it is necessary and sufficient that: \begin{enumerate}
\item $F$ is an Arakeljan set.
\item There exists a closed Arakeljan set $C_1\not=G$, so that $C\subset C_1$ and $F\cap C_1$ can be contained in a compact set in $G$.
\end{enumerate}
\end{thm}
\begin{proof}[Proof of Extension of Arakeljan's Theorem]

\textbf{(Necessity.)} Assume that any $f\in A(F)$ can be approximated uniformly by functions $g\in Hol(G)$, each of which is bounded on $C$. By Arakeljan's Theorem it follows $F$ is an Arakeljan set. Clearly the function $\phi(z)=z$ belongs to the class $Hol(G)$, thus there exists a function $f\in Hol(G)$ such that: \begin{equation}|z-f(z)|<1,\   \ z\in F  
\end{equation}and $f$ is also bounded on C.
 By {lemma 1} there exists a closed Arakeljan set $C_1\subset \mathbb{C}$ such that that $C\subset C_1$ and $f(z)$ is bounded on $C_1$. By $(1)$ clearly $F\cap C_1$ is bounded. Since $C_1$ is a closed subset inside $G$, it follows that it stays away from the boundary $\partial G$. Therefore, since $F\cap C_1$ is bounded it follows that it can be contained in a compact set $L\subset G$. The proof of necessity is complete.\\ 
 \noindent \textbf{(Sufficiency.)} We have Arakeljan sets $F$ and $C_1$, $C\subset C_1$, such that $F\cap C_1$ is a bounded set. Let $L$ be a compact set in $G$ such that $\partial K$ is the union of disjoint jordan curves and so large that:
 \begin{equation} 
  F\cap C_1\subset L\footnote{Such an L exists because $C_1$ is a closed subset in $G$. That is the reason we did not require $C_1$ to simply be a relatively closed subset of $G$ in the statement of the Theorem.}
   \end{equation}
  \begin{equation}
  (\partial L)\setminus C_1\not=\emptyset
   \end{equation}

\noindent Note that $(3)$ is possible since $C_1\not=G$ is a closed set without $G-$holes.\\
 The set $C_1\setminus L$ is also without $G-$holes, because otherwise there would exist a $G-$hole $g$ of $C_1\setminus L$. This leads to a contradiction since we must have $\partial g\subset C_1$. However, since $C_1$ has no $G-$holes it follows that $g=L$ and this contradicts $(3)$.\\
 Moreover, since by assumption of the theorem  $C$ and $F$ are $G-$hole independent, without loss of generality we may assume that the set $F\cup (C_1\setminus L)$ is without $G-$holes.
The set $C_1\setminus L$ is also an Arakeljan set. Otherwise, there would exist a compact set $K\subset G$ so that the set $H\equiv \lbrace h: h \text{ is a }G-\text{hole of }K\cup(C_1\setminus L)\rbrace$ is either unbounded or $\partial H\cap \partial G\not=\emptyset$. But this is a contradiction because then the set of $G-$holes of the set $C_1\setminus L$ union the compact set $K\cup L$ would also have to either be unbounded or accumulate to a point in $\partial G$. Indeed, the claim follows, since $(K\cup L)\cup (C_1\setminus L)=C_1\cup K$. This contradicts that $C_1$ is an Arakeljan set, as such, by the contradiction we conclude $C_1\setminus L$ is an Arakeljan set.\\
 Now the Arakeljan sets $F$ and $C_1\setminus L$ are disjoint by $(2)$, therefore by the lemma 4 it follows that $F\cup (C_1\setminus L)$ is an Arakeljan set. \\
 Let $\phi(z)\in A(F)$ be any function. Define a function $h(z)$ by $h(z)=\phi(z)$ on $F$ and $h(z)=0$ on $C_1\setminus L$. Clearly by $(2)$ it is evident that $h(z)\in A(F\cup (C_1\setminus L))$. Therefore by Arakeljan's theorem $h(z)$ is uniformly approximable by functions in $A(G)$. Therefore for any $\epsilon>0$, there exists a function $f\in A(G)$ such that $|h(z)-f(z)|<\epsilon$ for any $z\in F\cup (C_1\setminus L)$. Hence we have that $|\phi(z)-f(z)|<\epsilon$ on $F$ and $|f(z)|<\epsilon$ on $C_1\setminus L$. The function $f(z)$ is bounded on $C_1\setminus L$, moreover $f(z)$ is bounded on the compact set $C_1\cap L$ thus, $f(z)$ uniformly approximates $\phi(z)$ on $F$ and is bounded on $C$. The proof is complete.
\end{proof}

\noindent The following lemma is proven in [\cite{FU}, p2.65] and replaces lemma 4 for the case $G$ is simply connected.
\begin{lem}
Let $G\subset\mathbb{C}$ be a simply connected open set $\lbrace F_n\rbrace_{n=1}^\infty$ is a locally finite family of pairwise disjoint Arakeljan sets in $G$, then the union $\bigcup_{n=1}^\infty F_n$ is also an Arakeljan set in $G$.
\end{lem}

\begin{cor}[Extension of Arakeljan's Theorem when $G$ is simply connected]
Let $G\subset \mathbb{C}$ be an arbitrary simply connected domain and suppose $F$ is a relatively closed subset of $G$. Moreover assume that $C\subset G$ is a closed subset in $G$. In order that every function $f\in A(F)$ can be approximated uniformly on $F$ by functions $g\in Hol(G)$, each of which is bounded on $C$, it is necessary and sufficient that: \begin{enumerate}
\item $F$ is an Arakeljan set.
\item There exists a closed Arakeljan set $C_1\not=G$, so that $C\subset C_1$ and $F\cap C_1$ can be contained in a compact subset $L\subset G$.
\end{enumerate}
\end{cor}
\begin{proof}
The proof is identical to the one above where of course we use lemma 5 instead of lemma 4. 
\end{proof}

\noindent To conclude we ask the following open question that occurs naturally.
Does the Extension of Arakeljan's Theorem still hold if we require that the approximating functions to be uniformly bounded on $C$? This question remains unanswered even for the case $G=\mathbb{C}$

\pagebreak

\end{document}